\documentclass[twoside,a4paper,reqno,11pt]{amsart}
\usepackage[top=28mm,right=30mm,bottom=28mm,left=30mm]{geometry}

\usepackage{amsfonts, amsmath, amssymb, mathrsfs, bm, latexsym, stmaryrd, array, hyperref, mathtools, bold-extra, xcolor, tabularx, booktabs}

\newcommand{\CC}{\mathrm{C}}

\newcommand{\la}{\langle}
\newcommand{\ra}{\rangle}

\newcommand{\leqs}{\leqslant}

\newcommand{\Aut}{\mathrm{Aut}}

\newcommand{\PSL}{\mathrm{PSL}}
\newcommand{\GL}{\mathrm{GL}}

\newcommand{\PGammaL}{\mathrm{P\Gamma L}}

\newcommand{\PSU}{\mathrm{PSU}}

\newcommand{\Sz}{\mathrm{Sz}}

\def\Sym{\mathrm{Sym}}

\makeatletter
\newcommand{\imod}[1]{\allowbreak\mkern4mu({\operator@font mod}\,\,#1)}
\makeatother

\newtheorem{theorem}{Theorem}
\newtheorem*{conj*}{Conjecture}

\newtheorem{problem}[theorem]{Problem}

\newtheorem{example}[theorem]{Example}
\newtheorem{thm}{Theorem}[section]

\newtheorem{lem}[thm]{Lemma}
\newtheorem{cor}[thm]{Corollary}

\newtheorem*{prob*}{Problem}

\theoremstyle{definition}

\title{Fixers and Stabilizers for Ree groups}
\author{Yilin Xie }
\date{August 2024}

\begin{document}

\maketitle

\begin{abstract}
	Let $G$ be a finite permutation group on $\Omega,$ a subgroup $K\leqslant G$ is called a fixer if each element in $K$ fixes some element in $\Omega.$
	In this paper, we characterize fixers $K$ with $|K|\geqslant |G_\omega|$ for each primitive action of almost simple group $G$ with socle ${}^2G_2(q).$
\end{abstract}

\section{Introduction}

Let $G\leqslant \Sym(\Omega)$ be a finite transitive group.
A subgroup $K\leqslant G$ is called a \textit{fixer} if each element in $K$ fixes some point in $\Omega.$
The point stabilizer $G_\omega$ is a fixer, and $G_\omega$ may be a subgroup of a fixer in some cases.
We call a fixer $K$ \textit{stable} if $K$ is a subgroup of a point-stabilizer, as in \cite{HLX_fixer}.
Motivated by some topics of combinatorics, a natural question is to compare fixers and stabilizers, which led to the following definition:
\begin{itemize}
\item A fixer $K$ is called {\it maximal} if any subgroup $H$ with $K<H<G$ is not a fixer,
\item A fixer $K$ is called {\it maximum} if the order $|K|$ is the largest order among fixers.
\end{itemize}
Thus, for a maximum fixer $K$, we have $|K|\geqslant|G_\omega|$.
A fixer $K$ is called {\it large} if $|K|\geqslant|G_\omega|$.
A natural question is relative to the stabilizer: how large a large fixer can be?
In Example~\ref{ex:large-fixers}, the fixer has an order larger than the order of the stabilizer since the stabilizer is a proper subgroup of a fixer, which motivates us to explore the case where the stabilizer is a maximal subgroup, so the group primitive, refer to \cite{LPSX}.
On the other hand, Example~\ref{ex:small-fixers} shows that a maximal fixer may have a smaller order than the order of the point stabilizer; in this case, $K$ is called a {\it small maximal fixer}.
These observations motivate us to propose the following problem.
(Recall that a group $G$ is said to be {\it almost simple} if $T\trianglelefteqslant G\trianglelefteqslant \Aut(T)$ for a finite nonabelian simple group $T$.)

\begin{problem}\label{large-small}
{\rm
For an almost simple primitive permutation group $G$ on $\Omega$, 
\begin{itemize}
\item [(A)] determine large fixers,
\item [(B)] determine small maximal fixers.
\end{itemize}
}
\end{problem}

%

Problem~\ref{large-small}\,(A) has been solved for several notable cases, see \cite{HLX_fixer} for almost simple groups with socle $\PSL_2(q)$, and  \cite{LPSX} for almost simple Suzuki groups with socle $\Sz(q)$.
In this paper, we solve Problem~\ref{large-small}\,(A) for almost simple Ree groups.
With a subsequent paper which is devoted to solving Problem~\ref{large-small}\,(A) for almost simple groups with socle being $\PSU_3(q)$, Problem~\ref{large-small}\,(A) is solved for almost simple groups of Lie type of Lie rank 1.
Let
\[\mbox{$G_0={}^2G_2(q)$, where $q=3^{2n+1}$, a {\it Ree group}.}\]
Then $\Aut(G_0)=G_0{:}\la \phi \ra,$ where $\phi$ is a field automorphism of $G_0$ defined at the beginning of Section~\ref{sec:ree}.

The main result of this paper is the following theorem.

\begin{theorem}\label{thm:main}
Let $G_0={}^2G_2(q)$, where $q=3^{2n+1}$, and let $G=G_0.\langle \psi\rangle$ with $\psi\in\langle \phi\rangle$.
Let $G$ be a primitive permutation group on $\Omega$ which has a fixer $K$ with $|K|\geqslant|G_\omega|$.
Then either
\begin{itemize}
\item [{\rm(i)}] $G_\omega={}^2G_2(q_0).\langle \psi\rangle $, and $K=N_G(E_0)\cong ([q^3]{:}(q_0-1)){:}\langle \psi\rangle$, where $q=q_0^r$ for some odd prime $r\neq 3$, and $E_0$ is an elementary abelian $3$-group defined in Lemma~\ref{subfield} or

\item [{\rm(ii)}] $G={}^2G_2(27).\langle \psi\rangle $, $G_\omega=(2^2\times D_{14}){:}3{:}\langle \psi \rangle$, and
$K\cong (2^3{:}7{:}3){:}\langle \psi \rangle$ is the normalizer of a Sylow $2$-subgroup of $G$.



\end{itemize}
%
%
\end{theorem}

The principle motivations for the study of fixers in permutation groups come from the relations of fixers with dearrangements and intersecting subsets defined below.

One motivation is the study of EKR-type problems for permutation groups, which study intersecting subsets.
A subset $S\subseteq G$ is called \textit{intersecting} if for any two elements $x, y\in S,$ there exists some $\omega\in \Omega$ such that $\omega^x=\omega^y.$
It is not difficult to see $G_{\omega_1}g$ is an intersecting subset for any $\omega_1\in \Omega$ and $g\in G$ by definition. 
The group $G$ is said to have the Erd\"os-Ko-Rado property (abbreviated as EKR property) if for any intersecting subset $S$ of $G,$
$|S|\leqslant |G_\omega|,$ $G$ is further said to have the strict EKR property if for any intersecting subset $S$ with $|S|=|G_\omega|$, $S=G_{\omega_1}g$ for some $\omega_1\in \Omega$ and $g\in G$. 
EKR-type problems have been studied for a long time; refer to \cite{GM_EKR} and \cite{MST_2-trans}.

Given the definition of fixers and intersecting subsets, it is not difficult to see that a fixer is a subgroup of $G$ that is also an intersecting subset.
Recall that $G$ has the \textit{weak-EKR property} if $G$ does not have a fixer $K$ with $|K|>|G_\omega|,$ and $G$ has the \textit{strict-weak EKR property} if $G$ does not have a non-stable fixer $K$ with $|K|\geqslant |G_\omega|,$
(see the definition at \cite[page 2]{HLX_fixer}).
The following is a direct consequence of Theorem~\ref{thm:main}
\begin{cor}
    Suppose $G$ is a primitive permutation group on $\Omega$ with socle $G_0={}^2G_2(q),$ and $\omega \in \Omega.$
    Then 
    \begin{itemize}
        \item [{\rm(i)}] $G$ does not have the weak-EKR property if and only if $(G_0)_\omega$ is isomorphic to ${}^2G_2(q_0)$, where $q=q_0^r,$ and $r\neq 3$ is a prime;
        \item [{\rm(ii)}] If $G$ has the weak-EKR property, then $G$ has the strict-weak EKR property unless $q=27$ and $G_\omega$ is the normalizer of a Sylow $7$-subgroup of $G.$
    \end{itemize}
      
\end{cor}

Moreover, \cite[Theorem 2.10]{HLX_fixer} shows if $G$ is a primitive group and $G$ has a fixer $K$ with $|K|>|G_\omega|,$ then $G$ can be embedded in a wreath product $L\wr S_k$, where $k>1$, $L$ is an almost simple group acting on $\Sigma,$ and $\Omega=\Sigma^k$.
It is also shown in \cite[Lemma 3.3]{LPSX} that we can construct large fixers of $G$ from the large fixers of $L\leqs \mathrm{Sym}(\Sigma)$.
This also motivates us to study large fixers of almost simple primitive groups.

Another motivation comes from derangements of permutation groups. An element $g\in G$ is a derangement if $g$ does not fix any point in $\Omega,$ or $g\in G\backslash \bigcup_{g\in G}G_\omega^g$ in the language of abstract group theory. 
By a classical result of Jordan in \cite{J_der}, $G$ contains a derangement since $G$ is transitive.
Note that $K\leqslant G$ is a fixer if and only if $K$ does not contain any derangement of $G,$ so a fixer $K$ is not transitive.
Derangements of permutation groups have been studied for a long time. For example, it is determined in \cite{BG_classical} and \cite{BG_loc} the primitive actions of $G$ such that $G$ does not contain a derangement of order $p$ for all prime $p$ and classical groups $G.$
We refer to \cite{BG_classical} for more interesting problems about derangements of permutation groups.

Finally, let us briefly introduce how we establish Theorem~\ref{thm:main}.
Assume $G$ is an almost simple group with socle $G_0\cong {}^2G_2(q),$ and $K\leqslant G$ is a fixer with $|K|\geqslant |G_\omega|.$
We first introduce some basic properties of fixers in Section ~\ref{base:fixer}.
Then, we introduce some basic properties of the simple Ree group $G_0={}^2G_2(q)$ in Section~\ref{sec:ree}.
In particular, we collect its conjugacy classes of unipotent elements and subgroups and the representation of one of its Borel subgroups.
Then, for each primitive action of $G$, we examine all the subgroups of $G_0\cong {}^2G_2(q)$ and check whether they are fixers or not in Section~\ref{sec:red}. In particular, we establish Lemma~\ref{lem:reds}, which reduces the proof to two special cases.

To further deal with these two cases, we need specific results of the finite field $\mathbb{F}_q,$ which we record in Section~\ref{p:field}.
In particular, we collect some results of the trace map. We also describe a family of additive subgroups of $\mathbb{F}_q$ comprising multiples of the kernel of the trace map from $\mathbb{F}_q$ to $\mathbb{F}_3$, and we calculate their intersection, which is non-trivial.
This will be used to bound the order of a unipotent subgroup which intersects non-trivially with a unique $G_0$-class of elements of order $9,$ and eliminate the case  $q=q_0^3$ when $(G_0)_\omega \cong {}^2G_2(q_0)$ in Section~\ref{sec:proof}.
Finally, in Section~\ref{sec:proof}, we further deal with the two cases in Lemma~\ref{lem:reds}, we involve in the field automorphisms of $G_0$ and prove Theorem~\ref{thm:main}. 

\section{Basic properties for fixers}\label{base:fixer}
This section will collect some basic properties for fixers, most recorded in \cite{HLX_fixer}.

\begin{lem}\cite[Lemma 2.1]{HLX_fixer}\label{p:fixers} 
	The following statements are equivalent for a transitive permutation group $G\leqslant \Sym(\Omega)$
	\begin{itemize}
		\item [\rm(i)] $K\leqslant G$ is a fixer;
		\item [\rm(ii)] $K\subseteq \bigcup_{g\in G}G_\omega^g$;
		\item [\rm(iii)] For any $k\in K,$ there exists $g\in G$ such that $k^g\in G_\omega$;
		\item [\rm(iv)] $K^g$ is a fixer for some $g\in G$.
	\end{itemize}

\end{lem}

Next, we will give examples of two families of permutation groups and their fixers.

\begin{example}\label{ex:large-fixers}
{\rm
There are examples which show that a fixer may be arbitrarily larger than the stabilizer;
for instance, for $G=\PSL(2,p^f)$ and $\Omega=[G:H]$ with $|H|=p$, each Sylow $p$-subgroup of $G$ is a fixer.
}
\end{example}

\begin{example}\label{ex:small-fixers}
{\rm
 Let $G=\PSL_2(p)$ with $p$ prime and $G_\omega=D_{p-1}$ with $p\equiv 1$ $(\mod 15)$.
Then $A_5$ is a maximal fixer of order less than $p-1$ whenever $p\geqslant 67$.
}
\end{example}

Recall the \textit{spectrum} of a group $K$ is the set of all elements order in $K$, which we denote as $\mathrm{Spec}(K).$
The following lemma is a direct consequence of Lemma~\ref{p:fixers} \rm(iii).
\begin{lem}\label{lem:spec}
	Let $G$ be a transitive permutation group on $\Omega,$ if $K\leqslant G$ is a fixer, then $\mathrm{Spec}(K)\subseteq \mathrm{Spec}(G_\omega)$.
	
\end{lem}

Assuming further that $G$ is an almost simple primitive permutation group, the following helpful lemma allows us to first determine the large fixers of the socle of $G.$
\begin{lem}\label{lem:redu}
	Assume $G$ is an almost simple primitive permutation group on $\Omega,$ with a normal subgroup $G_0.$ 
	If $K$ is a fixer of $G$ with $|K|\geqslant |G_\omega|,$ then $K_0:=K\cap G_0$ is a fixer of $G_0$ with $|K_0|\geqslant |(G_0)_\omega|.$
\end{lem}
\begin{proof}
	Assume $K$ is a fixer of $G$ and $|K|\geqslant |G_\omega|$.
	Then each element in $K_0\leqslant K$ fixes some point in $\Omega,$ and hence $K_0$ is a fixer of $G_0.$
	Since 
	\begin{equation*}
		\frac{|K|}{|K_0|}\leqslant \frac{|G_\omega|}{|(G_0)_\omega|}=\frac{|G|}{|G_0|},
	\end{equation*}
	it follows that $|K_0|\geqslant |(G_0)_\omega|$.
	This completes the proof.
\end{proof}

\section*{Acknowledgements}
The author thanks professor Caiheng Li for his helpful comments on the earlier draft of this paper.

\section{Basic properties for finite Ree groups}\label{sec:ree}

The finite Ree groups $G_0={^2}G_2(q)$ over $\mathbb{F}_q,$ where $q=3^{2n+1}$ for some integer $n$ was first constructed by Rimhak Ree as the fixed point of an involutionary automorphism of Chevalley group of type $G_2$ over $\mathbb{F}_q$. Later, Robert Wilson gives a new construction using octonion algebra; see \cite{W_newcons}.
For the remainder of this paper, we adopt Rimhak Ree's construction of finite Ree groups.
In particular, let $\Phi$ be a root system of type $G_2$ with a fundamental system $\{a,b\}$, such that $a$ is of length $1$ and $b$ is of length $\sqrt{3}.$
Then the Chevalley elements $\{x_r(t)\mid r\in \Phi, r\in \mathbb{F}_q\}$ (see \cite[Page 63]{Ca_SGL} for the definition of $x_r(t)$) generate the Chevalley group $G_2(q).$ 
Let $r\mapsto \bar{r}, r\in \Phi$ be a symmetry of $\Phi$ such that 
\begin{equation*}
	\bar{a}=b,\quad \overline{a+b}=3a+b, \quad \overline{2a+b}=3a+2b, 
\end{equation*}
and let $\theta$ be the Frobenius automorphism of $\mathbb{F}_q$ such that $\lambda^\theta=\lambda^{3^n}$ for each $\lambda\in \mathbb{F}_q.$
Then 
\begin{equation}\label{e:tau}
 	x_r(t)\mapsto x_{\bar{r}}(t^{\theta(r,r)}), r\in \Phi, t\in \mathbb{F}_q 
\end{equation}
extends to an involutionary automorphism $\tau$ of $G_2(q)$, and $G_0={^2}G_2(q)$ is the subgroup of $G_2(q)$ comprising all the fixed elements of this involutionary automorphism.

The map 
\begin{equation*}
	x_r(t)\mapsto x_r(t^3), \quad r\in \Phi, t\in \mathbb{F}_q
\end{equation*}
can be extended to an automorphism $\phi$ of $G_2(q)$ (see \cite[Page 200]{Ca_SGL}).
Since $\phi$ commutes with $\tau$, for each $g\in G_0\cong {}^2G_2(q)$
\begin{equation*}
	g^{\phi\tau}=g^{\tau\phi}=(g^\tau)^\phi=g^\phi,
\end{equation*} 
and hence $g^\phi\in G_0.$
Thus, $\phi$ is a field automorphism of $G_0.$ 
Moreover, $\Aut(G_0)=G_0{:}\la \phi \ra$ (see \cite[Table 8.43]{low-dim}).

We next describe a Borel subgroup of $G_0,$ which is also a maximal subgroup of $G_0$ by \cite{Levchuk85} in the following lemma. 

\begin{lem}\label{Borel}
    Let 
    \begin{equation*}
    	\begin{aligned}
    		\alpha(t)&=x_a(t^\theta)x_b(t)x_{a+b}(t^{\theta+1})x_{2a+b}(t^{2\theta+1}),\\
    		\beta(u)&=x_{a+b}(u^\theta)x_{3a+b}(u),\\
    		\gamma(v)&=x_{2a+b}(v^{\theta})x_{3a+2b}(v).
    	\end{aligned}
    \end{equation*} 
    Then, 
    \begin{equation*}
    	Q=\{ \alpha(t)\beta(u)\gamma(v) \mid t, u, v\in \mathbb{F}_{3^{2n+1}} \}
    \end{equation*}
    is a unipotent subgroup of $G_0={^2}G_2(q).$
    We denote each element $\alpha(t)\beta(u)\gamma(v)$ in $Q$ as $X_S(t,u,v),$
    where $S=\{a,b,a+b,2a+b,3a+b,3a+2b\}$ is an equivalence class of $\Phi$ defined in \cite[Lemma 13.2.1]{Ca_SGL}.
    Let $n_r(t)=x_r(t)x_{-r}(-t^{-1})x_r(t),$ $h_r(t)=n_r(t)n_r(-1),$ and $h(s)=h_a(s)h_b(s^{3\theta}).$
    Then $$H=\{h(s), s\in \mathbb{F}_q^\times \}$$ is a subgroup of $G_0$ such that $B=QH$ is the normalizer of $Q$ in $G_0,$ with the following rules of calculation
    \begin{equation*}
    	\begin{aligned}
    		X_S(t_1,u_1,v_1)X_S(t_2,u_2,v_2)&=X_S(t_1+t_2,u_1+u_2-t_1t_2^{3\theta}, v_1+v_2-t_2u_1+t_1t_2^{3\theta+1}-t_1^2t_2^{3\theta}) \\
    		h(s_1)h(s_2)&=h(s_1s_2) \\
    		X_S(t,u,v)^{h(s)}&=X_S(s^{3\theta-2}t,s^{1-3\theta}u,s^{-1}v)
    	\end{aligned}
    \end{equation*}
\end{lem}
For the proof the above lemma, one may refer to page 236 of \cite{Ca_SGL}, or \cite{matrixgen} alternatively.

We next describe the basic group operation on elements in the unipotent subgroup of $G_0$ defined in the above Lemma~\ref{Borel}.
\begin{lem}\label{uni}
	    The following properties hold for $Q$.
	\begin{itemize}
		\item[{\rm(i)}] 
		$X_S(t,u,v)^{-1}=X_S(-t,-(u+t^{3\theta+1}),-(v+tu)+t^{3\theta+2})$;
		\item[{\rm(ii)}]  $X_S(t,u,v)^3=X_S(0,0,-t^{3\theta+2})$;
		\item [{\rm(iii)}] 
		$X_S(t,u,v)^{X_S(t',u',v')}=X_S(t,u-t(t')^{3\theta}+(t')t^{3\theta},v-t'u+tu'+t(t')^{3\theta+1}-t't^{3\theta+1}-t^2(t')^{3\theta}+(t')^2t^{3\theta});$
		\item[{\rm(iv)}] $[X_S(t',u',v'),X_S(t,u,v)^{-1}]=X_S(0,-t(t')^{3\theta}+t't^{3\theta},-t'u+tu'+t(t')^{3\theta+1}+t^2(t')^{3\theta}+(t')^2t^{3\theta});$
		\item[\rm(v)]
		$Z(Q)=\{X_S(0,0,v)\mid v\in \mathbb{F}_{3^{2n+1}}\};$
		\item[\rm(vi)]
		$Q'=\{X_S(0,u,v)\mid u,v\in \mathbb{F}_{q} \}$ is an elementary abelian group of order $q^2,$ and $Z(Q)\leqslant Q';$
		
	\end{itemize}
\end{lem}
\begin{proof}
By the rules of multiplication on $Q$ in Lemma~\ref{Borel} and a direct calculation, one get \rm(i) $\sim$ \rm(iv).
\rm(v) follows directly from \rm(iii) and a direct calculation.
For \rm(vi), it follows from (2) of the main theorem of \cite{W_ree} that 
\begin{equation*}
	|Q'|=q^2=|\{X_S(0,u,v)\mid u,v\in \mathbb{F}_q\}|
\end{equation*}
On the other hand, \rm(iv) implies that $Q'\leqslant \{X_S(0,u,v)\mid u,v\in \mathbb{F}_q \},$ yielding \rm(vi).
\end{proof}

The next lemma can be deduced from the main result of \cite{matrixgen}, which will be useful.
\begin{lem}\label{7dimrep}
	There is a $7$-dimensional representation of $G_0$ over $\mathbb{F}_q$ $\rho: G_0 \longrightarrow \GL_7(q)$.
	Under this representation, the Borel subgroup of $G_0$ defined in Lemma~\ref{Borel} coincides with all upper-triangular matrices in $\rho(G_0)$.
    Moreover, 
    \begin{equation*}
        \rho(g^\phi)=\rho(g)^\phi,
    \end{equation*}
    for all $g\in G_0,$
    where \[(\rho(g)^{\phi})_{i,j}=(\rho(g)_{i,j})^{\phi}=\rho(g)_{i,j}^3,\] and $\phi$ is the standard Frobenius automorphism.
\end{lem}

We next analyze the $G_0$-classes of unipotent elements, we begin with the following lemmas.

\begin{lem}\label{centuni}
	The centralizer of the elements in $Q$ are calculated as follows
	\begin{itemize}
		\item [\rm(i)] If $x\in Z(Q),$ then $C_{G_0}(x)=Q$, in particular, $|C_{G_0}(x)|=q^3$;
		\item [\rm(ii)] If $x \in Q'\backslash Z(Q),$ then $|C_{G_0}(x)|=2q^2$;
		\item [\rm(iii)] If $x\in Q\backslash Q',$ then $C_{G_0}(x)=\la x \ra Z(Q),$ in particular, $|C_{G_0}(x)|=3q$; 
	\end{itemize}
\end{lem}
\begin{proof}
	Let $B$ be the Borel subgroup of $G_0$ defined in Lemma~\ref{Borel}.
	By \cite[Lemma 1]{Levchuk85}, $C_{G_0}(x)=C_B(x).$
	
	Assume first $x\in Z(Q).$
	Then Lemma~\ref{Borel} implies that $ C_H(x)=1,$ yielding $C_{G_0}(x)=QC_H(x)=Q.$
	
	To see \rm(ii), note that  $C_Q(x)=Q'$ by Lemma~\ref{uni} \rm(iii) and a direct calculation.
	Thus, $C_{G_0}(x)=Q'{:}H_1$ for some $H_1\leqslant B$ with order coprime to $3$ by Hall's theorem. 
	Applying Hall's theorem again, there exists $q\in Q$ such that $H_1^q\leqslant H.$ 
	It follows that
	\begin{equation*}
		C_{G_0}(x^q)=C_{G_0}(x)^q=Q'H_1^q=Q'C_H(x^q).
	\end{equation*}
	On the other hand, $C_H(x^q)=\la h(-1) \ra$ by Lemma~\ref{Borel}, yielding \rm(ii).
	
	We finally assume $x\in Q\backslash Q'.$
	Note that $3\theta-2$ is coprime to $|\mathbb{F}_q^\times|=3\theta^2-1,$ it follows that $B/Q'\cong \mathrm{AGL}_1(q).$ 
	Thus, $C_{B/Q'}(xQ')=Q/Q',$ which implies $C_{G_0}(x)\leqslant Q.$
	Let $y\in C_Q(x)$ and assume $x=X_S(t,u,v)$ and $y=X_S(t',u',v').$
	Thus, $u=u-t(t')^{3\theta}+t't^{3\theta}$ by Lemma~\ref{uni} \rm(iii), this is equivalent to
	 \begin{equation*}
	 	-t(t')^{3\theta}+t't^{3\theta}=t^{3\theta+1}(\frac{t'}{t}-(\frac{t'}{t})^{3\theta})=0.
	 \end{equation*}
	Note that $t\in \mathbb{F}_q^{\times},$ it follows that $\frac{t'}{t}=(\frac{t'}{t})^{3\theta},$ yielding $t'\in \{0,t,-t\}.$
	Lemma~\ref{uni} \rm(iii) also implies 
	$v-t'u+tu'+t(t')^{3\theta+1}-t't^{3\theta+1}-t^2(t')^{3\theta}+(t')^2t^{3\theta}=v.$
	Thus, for each for each $t'\in \{0,t,-t\},$ there is a unique $u'\in \mathbb{F}_q$ such that $x^{X_S(t',u',v')}=x.$
	Therefore, $|C_{G_0}(x)|=|C_Q(x)|\leqslant 3q$.
	On the other hand, $\la x \ra Z(Q)\leqslant C_{G_0}(x)$ and $|\la x \ra Z(Q)|=3q$ by Lemma~\ref{uni} \rm(ii), this implies $C_{G_0}(x)=\la x \ra Z(Q).$
	This completes the proof of \rm(iii).
	\end{proof}

    \begin{lem}\label{conj9}
     Let $x=X_S(t,u,v)\in Q\backslash Q',$ then $x$ is conjugate to $x^{-1}$ in $G_0$ if and only if there exists $t'\in \mathbb{F}_q$ such that
     $u-t^{3\theta+1}=-t(t')^{3\theta}+t't^{3\theta}.$  
    \end{lem}
    \begin{proof}
     Assume there exists $g\in G_0$ such that $x^g=x^{-1}.$
     Then $g\in N_{G_0}(\la x \ra)\leqslant N_{G_0}(Q)=B$ by \cite[Lemma 1]{Levchuk85}, and hence $g=X_S(t',u',v')h(s^{-1})$ for some $t',u',v' \in \mathbb{F}_q$ and $s\in \mathbb{F}_q^{\times}.$
     Thus,
     \begin{equation*}
     	\begin{aligned}
     	x^{X_S(t',u',v')}&=X_S(t,u-t(t')^{3\theta}+(t')t^{3\theta},v-t'u+tu'+t(t')^{3\theta+1}-t't^{3\theta+1}-t^2(t')^{3\theta}+(t')^2t^{3\theta})\\
     	&=(x^{-1})^{h(s^{-1})}=X_S(-s^{2-3\theta}t,-s^{3\theta-1}(u+t^{3\theta+1}),s(-(v+tu)+t^{3\theta+2}))
     	\end{aligned}
     \end{equation*} 
    by Lemma~\ref{uni} \rm(iii) and Lemma~\ref{Borel}.
    Thus, $t=-s^{2-3\theta}t$ and $s^{2-3\theta}=-1,$ and hence $s=-1$ by noting that $2-3\theta$ is coprime to $|\mathbb{F}_q^{\times}|.$
    Therefore, $u-t^{3\theta+1}=-t(t')^{3\theta}+t't^{3\theta}.$
    
    On the other hand, if there exists $t'\in \mathbb{F}_q$ such that $u-t^{3\theta+1}=-t(t')^{3\theta}+t't^{3\theta}.$
    Then we take $u'\in \mathbb{F}_q$ such that 
    \begin{equation*}
    	v-t'u+tu'+t(t')^{3\theta+1}-t't^{3\theta+1}-t^2(t')^{3\theta}+(t')^2t^{3\theta}=v+tu-t^{3\theta+2},
    \end{equation*}
   we remark that $u'$ exists since $t\in \mathbb{F}_q^{\times}.$
   Therefore, $x^{X_S(t',u',v')}=(x^{-1})^{h(-1)},$ this completes the proof.
    \end{proof}

   We are now ready to determine the $G_0$-classes of unipotent elements.
   Recall the trace map from $\mathbb{F}_q$ to one of its subfield $\mathbb{F}_{q_0},$ where $q=q_0^r$
   $$\mathrm{Tr}_{\mathbb{F}_q/\mathbb{F}_{q_0}}(x)=x+x^{q_0}+\dots+x^{q_0^{r-1}}.$$
   We also recall the additive form of Hilbert 90, which says for any $s$ coprime to $r,$ there exists $\mu \in \mathbb{F}_q$ such that $\lambda=\mu^{q_0^s}-\mu$ if and only if $\mathrm{Tr}_{\mathbb{F}_q/\mathbb{F}_{q_0}}(\lambda)=0.$  
   
   \begin{lem}\label{conj_uni}
   	Let $u\in \mathbb{F}_q$ such that $\mathrm{Tr}_{\mathbb{F}_q/\mathbb{F}_3}(u-1)\neq 0,$ and let $x_1=X_S(1,u,0)\in Q.$
   	Let $x_2=X_S(1,1,0),$ $y=X_S(0,1,0)$ and $z=X_S(0,0,1)$ be three elements in $Q.$
   	Then 
   	\begin{equation*}
   		x_1^{G_0}, \, (x_1^{-1})^{G_0},\, x_2^{G_0},\, y^{G_0},\, (y^{-1})^{G_0},\, z^{G_0}
   	\end{equation*}
    are all ${G_0}$-classes of unipotent elements.
   \end{lem}
   \begin{proof}
   	Note that there does not exist $t'\in \mathbb{F}_q$ such that $u-1=t'-(t')^{3\theta},$ $x_1$ is not conjugate to $(x_1)^{-1}$ by Lemma~\ref{conj9}.
   	
   	We claim that $y$ is not conjugate to $y^{-1}.$ 
   	Assume the contrary there exists $g\in {G_0}$ such that $y^g=y^{-1}.$
   	Then $g\in N_{G_0}(\la y \ra)\leqslant N_{G_0}(Q)=B$ by \cite[Lemma 1]{Levchuk85}.
   	Thus, $g=X_S(t,u,v)h(s)$ for some $t,u,v \in \mathbb{F}_q$ and $s\in \mathbb{F}_q^\times$, and hence 
   	\begin{equation*}
   		X_S(0,1,0)^{X_S(t,u,v)h(s)}=X_S(0,s^{1-3\theta},s^{-1}(v-t'))=X_S(0,-1,0).
   	\end{equation*} 
   However, $s^{1-3\theta}\neq -1$ by noting that $-1$ is a non-square elements in $\mathbb{F}_q^\times,$ a contradiction. This proves the claim.
   
   Thus, $x_1, x_1^{-1}, x_2, y, y^{-1}$ and $z$ are pairwisely not conjugate to each other in $G_0$ by Lemma~\ref{centuni}. 
   Moreover, each unipotent element in $G_0$ is conjugate to one of $x_1, x_1^{-1}, x_2, y, y^{-1}$ and $z$ by \cite[Table 22.2.7]{LS_uninil}. This completes the proof.
   	\end{proof}

\subsection*{Subgroups}
We collect some known results of the subgroups of $G_0$ here.

\begin{lem}\cite[Page 5]{Levchuk85} \label{Hall}
	The group $G_0\cong {}^2G_2(q)$ has the following four cyclic Hall-subgroups
	\begin{itemize}
		\item [\rm(i)] $A_0\cong \CC_{\frac{q-1}{2}},$ and $N_{G_0}(A_0)\cong D_{2(q-1)}$;
		\item [\rm(ii)] $A_1\cong \frac{q+1}{4},$ and $N_{G_0}(A_1)=T{:}(A_1{:}\la t \ra),$ where $T\cong \CC_2^2,$ and $t$ is of order $6$;
		\item[\rm(iii)] $A_2\cong \CC_{q+1-\sqrt{3q}},$ and $N_{G_0}(A_2)\cong \CC_{q+1-\sqrt{3q}}{:}\CC_6$;
		\item[\rm(iv)] $A_3 \cong \CC_{q+1+\sqrt{3q}},$ and $N_{G_0}(A_3)\cong \CC_{q+1+\sqrt{3q}}{:}\CC_6$.
		\end{itemize}
\end{lem}

\begin{lem}\cite[Table 8.43]{low-dim} \label{max_subgps}
	The maximal subgroups of $G_0$ are conjugate to one of the following subgroups:
	\begin{itemize}
		\item [\rm(i)] The Borel subgroup defined in \ref{Borel}.
		\item [\rm(ii)] $G_{G_0}(\eta)\cong \CC_2\times \PSL_2(q),$ where $\eta$ is an involution;
		\item [\rm(iii)] $C_{G_0}(\phi^r)\cong {}^2G_2(q_0),$ where $q=q_0^r$;
		\item [\rm(iv)] $N_G(A_i),$ where $1\leqslant i \leqslant 3$, and $A_i$ are the cyclic Hall subgroups defined in Lemma~\ref{Hall},
	\end{itemize}
and the maximal subgroups of $\Aut(G_0)$ are extension of the maximal subgroups of $G_0$ by field automorphisms.
\end{lem}

\begin{lem}\cite[Page 5]{Levchuk85}\label{normsyl2}
	A Sylow $2$-subgroup of $G_0$ is elementary abelian of order $8$, and its index in normalizer is $21$.
\end{lem}

\begin{lem}\cite[Lemma 2]{Levchuk85}\label{solgps}
 A soluble subgroup of $G_0$ is conjugate to a subgroup of the following groups
 \begin{equation*}
 	N_{G_0}(A_i), \, 0\leqslant i \leqslant 4, \quad N_{G_0}(S_p), \, p\in \{2,3\},
 \end{equation*} 
where $A_i, 0\leqslant i \leqslant 4$ are the cyclic Hall subgroups defined in Lemma~\ref{Hall}, and $S_p$ is a Sylow $p$-subgroups of $G_0.$
\end{lem}

\begin{lem}\cite[Lemma 6]{Levchuk85} \label{nonsol}
	An insoluble subgroup of $G_0$ is isomorphic to one of the following

\begin{equation*}
	\PSL_2(8), \PSL_2(q'), \eta\times \PSL_2(q'), {}^2G_2(q'),
\end{equation*}
where $\eta$ is an involution and $q=(q')^k$ for some integer $k.$
\end{lem}

\section{Some property of finite field $\mathbb{F}_q$}\label{p:field}
In this section, we record some properties of finite field $\mathbb{F}_q.$
In particular, we study a family of additive subgroups of $\mathbb{F}_q$ and their intersection.
We also state some more properties of trace maps of finite fields.

Let $A$ be the set $\{\alpha-\alpha^{3\theta}\mid \alpha\in \mathbb{F}_q\},$ then $A$ is the kernel of the trace map from $\mathbb{F}_q$ to $\mathbb{F}_3$ by additive form of Hilbert 90.
Let $tA=\{ta\mid a\in A\},$ then $tA$ is an $\mathbb{F}_3$-subspace of $\mathbb{F}_q.$
Assume $T$ is a subset of $\mathbb{F}_q^\times.$
We study the intersection $\bigcap\limits_{t\in T}tA.$
This will be useful when we study a subgroup of $Q$ that intersects non-trivially with a unique $G$-class of elements of order $9$ in Section~\ref{sec:proof}.   

We first make an observation in the following lemma.
\begin{lem}\label{lem:intA}
	Let $T^{-1}=\{t^{-1}\mid t\in T\},$ and let $\mathrm{span}_{\mathbb{F}_3}(T^{-1})$ be the $\mathbb{F}_3$-subspace of $\mathbb{F}_q$ generated by all elements in $T^{-1}.$
	Then 
	\begin{equation}\label{e:intA}
	\bigcap_{t\in T}tA=\bigcap_{t\in \mathrm{span}_{\mathbb{F}_3}(T^{-1})}t^{-1}A.
	\end{equation}
\end{lem}
\begin{proof}
	It is clear that $\bigcap\limits_{t\in \mathrm{span}_{\mathbb{F}_3}(T^{-1})}t^{-1}A \subseteq \bigcap\limits_{t\in T}tA,$ and hence it suffices to show $\bigcap\limits_{t\in T}tA \subseteq \bigcap\limits_{t\in \mathrm{span}_{\mathbb{F}_3}(T^{-1})}t^{-1}A.$ 
	We prove it by induction on $|T|.$
	
	If $T$ has a unique element, then (\ref{e:intA}) holds immediately.
	
	Assume $T=\{t,t_1\}.$ If $t=-t_1$, then (\ref{e:intA}) holds immediately.
	We thus assume $t_1\notin \{\pm t\}.$
	Then 
	\begin{equation}\label{e:int2A}
		\begin{aligned}
		tA\cap t_1A &=\{t(\alpha-\alpha^{3\theta})\mid \alpha\in \mathbb{F}_q\}\cap t_1A\\
		&=\{t_1(\frac{t}{t_1}\alpha-\frac{t}{t_1}\alpha^{3\theta})\mid \alpha\in \mathbb{F}_q\}\cap t_1A\\
		&=\{t_1(\frac{t}{t_1}\alpha-\frac{t^{3\theta}}{t_1^{3\theta}}\alpha^{3\theta}+\alpha^{3\theta}(\frac{t^{3\theta}}{t_1^{3\theta}}-\frac{t}{t_1}))\mid \alpha \in \mathbb{F}_q\}\cap t_1A\\
		&=\{t(\alpha-\alpha^{3\theta})\mid \alpha^{3\theta}\in \frac{1}{(\frac{t}{t_1})^{3\theta}-\frac{t}{t_1}}A \},
		\end{aligned}
	\end{equation} 
    the last equality holds by noting that $t_1(\frac{t}{t_1}\alpha-\frac{t^{3\theta}}{t_1^{3\theta}}\alpha^{3\theta})\in t_1A.$
    On the other hand, by (\ref{e:int2A}), for any $a,b\in \mathbb{F}_3^\times=\{\pm 1\}$ we have
    \begin{equation*}
    	\begin{aligned}
    	tA\cap \frac{1}{\frac{a}{t_1}+\frac{b}{t}}A&=\{t(\alpha-\alpha^{3\theta})\mid \alpha^{3\theta}\in \frac{1}{t^{3\theta}(\frac{a}{t_1}+\frac{b}{t})^{3\theta}-t(\frac{a}{t_1}+\frac{b}{t})}A \}\\
    	&=\{t(\alpha-\alpha^{3\theta})\mid \alpha^{3\theta}\in \frac{1}{(\frac{t}{t_1})^{3\theta}-\frac{t}{t_1}}A \}=tA\cap t_1A.
    \end{aligned}
    \end{equation*}
    Therefore, $tA\cap t_1A\subseteq \frac{1}{\frac{a}{t_1}+\frac{b}{t}}A$, and (\ref{e:intA}) holds.
    
    Suppose (\ref{e:intA}) holds when $|T|\leqslant k-1,$ where $k\geqslant 3,$ we next prove it for $|T|=k.$ 
    We thus assume $T=\{t_1,t_2,\dots, t_k\}.$ 
    Let $\lambda$ be an element in $\mathrm{span}_{\mathbb{F}_3}(T^{-1}),$ we need to show $\bigcap\limits_{t\in T}tA\subseteq \frac{1}{\lambda}A.$
    Indeed, if $\lambda\in \mathrm{span}_{\mathbb{F}_3}\{t_1^{-1},t_2^{-1},\dots, t_{k-1}^{-1}\},$ then $$\bigcap_{t\in T}tA\subseteq \bigcap_{i=1}^{k-1}t_iA \subseteq \frac{1}{\lambda}A $$ 
    by induction. 
    We thus may assume $\lambda\notin \mathrm{span}_{\mathbb{F}_3}\{t_1^{-1}, t_2^{-1}, \dots, t_{k-1}^{-1}\},$ then there exists $1 \leqslant i_1, i_2, \dots, i_j \leqslant k-1$ and $a_1, a_2, \dots, a_{j+1}\subseteq \mathbb{F}_3^\times,$ such that $\lambda=a_1t_{i_1}^{-1}+\dots +a_jt_{i_j}^{-1}+a_{j+1}t_k^{-1}.$
    Thus,
    \begin{equation*}
    	\begin{aligned}
    	\bigcap_{t\in T}tA &\subseteq (\bigcap_{i=1}^{j}t_{i_j}A) \cap t_kA\\
    	&\subseteq \frac{1}{a_1t_{i_1}^{-1}+\dots +a_jt_{i_j}^{-1}}A \cap t_kA \subseteq \frac{1}{\lambda}A, 
    	\end{aligned}
    \end{equation*}
 by induction. 
 Therefore, $\bigcap\limits_{t\in T}tA \subseteq \bigcap\limits_{t\in \mathrm{span}_{\mathbb{F}_3}(T^{-1})}t^{-1}A$ since $\lambda$ is taken arbitrarily in $\mathrm{span}_{\mathbb{F}_3}(T^{-1}).$
 This completes the proof.
	\end{proof}

Recall that $q=3^{2n+1}$ and $\mathbb{F}_q$ is a vector space over $\mathbb{F}_3$ of dimension $2n+1,$ and $\bigcap\limits_{t\in T}tA$ is an $\mathbb{F}_3$-subspace of $\mathbb{F}_q.$  The following lemma describes the size of $\bigcap\limits_{t\in T}tA.$
\begin{lem}\label{lem:dimintA}
	\begin{equation}\label{dimintA}
		\mathrm{dim}_{\mathbb{F}_3}(\bigcap_{t\in T}tA)+\mathrm{dim}_{\mathbb{F}_3}(\mathrm{span}_{\mathbb{F}_3}(T^{-1}))=2n+1
	\end{equation}
\end{lem}    
\begin{proof}
	We prove (\ref{dimintA}) by induction on the dimension of $\mathrm{span}_{\mathbb{F}_3}(T^{-1})$ over $\mathbb{F}_3.$
	If $\mathrm{span}_{\mathbb{F}_3}(T^{-1})$ is of dimension $1,$ then $T$ is contained in a one-dimensional $\mathbb{F}_3$ subspace, and hence (\ref{dimintA}) holds immediately.
	
	We now assume (\ref{dimintA}) holds when $\mathrm{dim}_{\mathbb{F}_3}(\mathrm{span}_{\mathbb{F}_3}(T^{-1}))\leqslant k-1,$ and prove it when $\mathrm{dim}_{\mathbb{F}_3}(\mathrm{span}_{\mathbb{F}_3}(T^{-1}))=k.$
	We thus assume $\{t_1^{-1}, t_2^{-1}, \dots, t_k^{-1}\}$ is a basis for $T^{-1},$ where $t_1, t_2, \dots t_k$ are $k$ distinct elements in $T.$
	Note that $$\mathrm{dim}_{\mathbb{F}_3}(\bigcap\limits_{t\in T}tA)=\mathrm{dim}_{\mathbb{F}_3}(\bigcap\limits_{t\in t_0T}tA)$$
	for any $t_0\in \mathbb{F}_q$, we may replace $T$ with $t_0T$ for a suitable $t_0\in \mathbb{F}_q,$ and assume further $t_1=1\in T.$
	Then 
	\begin{equation*}
		\bigcap_{t\in T}tA=\bigcap_{i=1}^{k}t_iA=\bigcap_{i=2}^{k}(A\cap t_iA),
	\end{equation*}
   by Lemma~\ref{lem:intA}.
   Note that
   \begin{equation*}
   	\begin{aligned}
   	\bigcap_{i=2}^{k}(A\cap t_iA)&=\bigcap_{i=2}^{k}\{\alpha-\alpha^{3\theta}\mid \alpha^{3\theta}\in \frac{1}{(\frac{1}{t_i})^{3\theta}-\frac{1}{t_i}}A\}\\
   	 &=\{\alpha-\alpha^{3\theta}\mid \alpha^{3\theta}\in \bigcap_{i=2}^k\frac{1}{(\frac{1}{t_i})^{3\theta}-\frac{1}{t_i}}A \}
   	 \end{aligned}
   \end{equation*}
   by (3) in the proof of Lemma~\ref{lem:intA}.
   We also note that $1\notin \mathrm{span}_{\mathbb{F}_3}\{t_2^{-1},\dots, t_k^{-1}\},$ thus the linear map which sends each $\lambda\in \mathbb{F}_q$ to $\lambda-\lambda^{3\theta}$ with kernel $\mathbb{F}_3$ is injective on $\mathrm{span}_{\mathbb{F}_3}\{t_2^{-1},\dots, t_k^{-1}\},$ yielding 
   \begin{equation*}
   	\mathrm{dim}_{\mathbb{F}_3}(\mathrm{span}_{\mathbb{F}_3}\{(\frac{1}{t_2})^{3\theta}-\frac{1}{t_2},(\frac{1}{t_3})^{3\theta}-\frac{1}{t_3},\dots, (\frac{1}{t_k})^{3\theta}-\frac{1}{t_k}\})=k-1.
   \end{equation*}
   On the other hand, $(\frac{1}{t_i})^{3\theta}-\frac{1}{t_i}\in A$ for each $2\leqslant i \leqslant k,$ thus $$\{\alpha\mid \alpha-\alpha^{3\theta}=0 \}=\mathbb{F}_3\subseteq \bigcap\limits_{i=2}^{k}\frac{1}{(\frac{1}{t_i})^{3\theta}-\frac{1}{t_i}}A.$$
   Therefore, 
   \begin{equation*}
   	\mathrm{dim}_{\mathbb{F}_3}(\bigcap_{i=1}^{k}t_iA)=\mathrm{dim}_{\mathbb{F}_3}(\bigcap\limits_{i=2}^{k}\frac{1}{(\frac{1}{t_i})^{3\theta}-\frac{1}{t_i}}A)-1=2n+1-k
   \end{equation*} 
   by induction. This completes the proof.
	\end{proof}

The following lemma will be useful in Section~\ref{sec:proof} when we prove Theorem~\ref{thm:main}.
\begin{lem}\label{trqq1}
	Let $q=q_0^r$ with $r\neq 3$ an odd prime, then for any subfield $\mathbb{F}_{q_1}$ of $\mathbb{F}_q$, there exists $\lambda \in \mathbb{F}_{q_0}$ such that  $\mathrm{Tr}_{\mathbb{F}_q/\mathbb{F}_{q_1}}(\lambda)\in \mathbb{F}_{q_1}^\times.$
	In particular, there exists $\lambda_1\in \mathbb{F}_{q_0}$ such that $\mathrm{Tr}_{\mathbb{F}_q/\mathbb{F}_{q_1}}(\lambda_1)=1.$
\end{lem}
\begin{proof}
	Let $\mathbb{F}_{q_1}$ be a subfield of $\mathbb{F}_q.$
	We first claim that there exists 
	$\lambda\in \mathbb{F}_q$ such that $\mathrm{Tr}_{\mathbb{F}_q/\mathbb{F}_{q_1}}(\lambda)\neq 0.$
	Assume $\mathbb{F}_{q_1}\subseteq \mathbb{F}_{q_0}.$
	Then we take $\lambda\in \mathbb{F}_{q_0}$ such that $\mathrm{Tr}_{\mathbb{F}_{q_0}/\mathbb{F}_{q_1}}(\lambda)\neq 0,$ and hence 
	\begin{equation*}
		\mathrm{Tr}_{\mathbb{F}_q/\mathbb{F}_{q_1}}(\lambda)=\mathrm{Tr}_{\mathbb{F}_{q_0}/\mathbb{F}_{q_1}}(\mathrm{Tr}_{\mathbb{F}_q/\mathbb{F}_{q_0}}(\lambda))=\mathrm{Tr}_{\mathbb{F}_{q_0}/\mathbb{F}_{q_1}}(r\lambda)=r\mathrm{Tr}_{\mathbb{F}_{q_0}/\mathbb{F}_{q_1}}(\lambda)\neq 0.
	\end{equation*}
It remains to consider the case $\mathbb{F}_{q_1}$ is not a subfield of $\mathbb{F}_{q_0}.$
Then $\mathbb{F}_q=\mathbb{F}_{q_1}\mathbb{F}_{q_0}.$
Assume the contrary that $\mathrm{Tr}_{\mathbb{F}_q/\mathbb{F}_{q_1}}(\lambda)=0$ for all $\lambda\in \mathbb{F}_{q_0}.$
Then  $\mathrm{Tr}_{\mathbb{F}_q/\mathbb{F}_{q_1}}(\mu)=0$ for any $\mu\in \mathbb{F}_{q}$, a contradiction. 
This proves the claim.

By the previous paragraph, there exists $\lambda\in \mathbb{F}_{q_0}$ such that $\mathrm{Tr}_{\mathbb{F}_q/\mathbb{F}_{q_1}}(\lambda)\in \mathbb{F}_{q_1}^\times.$
Note that $\mathrm{Tr}_{\mathbb{F}_q/\mathbb{F}_{q_1}}(\lambda)\in \mathbb{F}_{q_0},$ so $\mathrm{Tr}_{\mathbb{F}_q/\mathbb{F}_{q_1}}(\lambda)^{-1}\lambda\in \mathbb{F}_{q_0},$ and
\begin{equation*}
	\mathrm{Tr}_{\mathbb{F}_q/\mathbb{F}_{q_1}}(\mathrm{Tr}_{\mathbb{F}_q/\mathbb{F}_{q_1}}(\lambda)^{-1}\lambda)=\mathrm{Tr}_{\mathbb{F}_q/\mathbb{F}_{q_1}}(\lambda)^{-1}\mathrm{Tr}_{\mathbb{F}_q/\mathbb{F}_{q_1}}(\lambda)=1.
\end{equation*} 
This completes the proof.
	\end{proof}
\section{A reduction}\label{sec:red}

In this section, we first examine all the subgroups of $G_0={}^2G_2(q)$ and use Lemma~\ref{lem:redu} to reduce the possibility of $G_\omega$ and fixer $K$ with $|K|\geqslant |G_\omega|$ to two cases. 
In particular, we prove the following lemma.
\begin{lem}\label{lem:reds}
	Assume $G$ is an almost simple group with socle $G_0\cong {}^2G_2(q),$ and $K\leqslant G$ is a fixer with $|K|\geqslant |G_\omega|.$
	Then, one of the following holds
	\begin{itemize}
		\item [\rm(i)] $G_\omega$ is a maximal subfield subgroup with $(G_0)_\omega\cong {}^2G_2(q_0),$ where $q=q_0^r$ and $r$ is an odd prime, and $K$ is contained in a normalizer of a Sylow $3$-subgroup;
		\item [\rm(ii)] $q=27,$ $G_\omega$ is the normalizer of a cyclic Sylow $7$-subgroup of order $7,$ and $K$ is the normalizer of a Sylow $2$-subgroup of $G,$ further in this case if $G=G_0,$ then $K$ is indeed a fixer with $|K|=|G_\omega|$. 
	\end{itemize} 
\end{lem} 

\begin{lem}\label{f:borel}
	If $G_\omega$ is a normalizer of some Sylow $3$-subgroup of $G_0$, then $G$ does not have a non-stable fixer $K$ with $|K|\geqslant |G_\omega|.$ 
\end{lem}
\begin{proof}
	Since $G$ acts $2$-transitively on $\Omega,$ this action has the EKR property by the main theorem of \cite{MST_2-trans}, and hence $G$ does not have a non-stable fixer $K$ with $|K|\geqslant |G_\omega|.$
	This completes the proof. 
	\end{proof}

\begin{lem}\label{f:2cent}
	If $G_\omega$ is a centralizer of an involution, then $G$ does not have a non-stable fixer $K$ with $|K|\geqslant |G_\omega|.$
\end{lem}
\begin{proof}
	Assume $K$ is a fixer of $G$ and $|K|\geqslant |G_\omega|.$
	Then $K_0$ is a fixer of $G_0$
	and $|K_0|\geqslant |(G_0)_\omega|$ by Lemma~\ref{lem:redu}. 
	Thus, $K_0$ is a subgroup of a Borel subgroup by computing the order of all maximal subgroups of $G_0$ in Lemma~\ref{max_subgps}.
	In view of Lemma~\ref{p:fixers} \rm(iv), we may replace $K$ with $K^g$ for some $g\in G_0$ and assume $K_0$ is a subgroup of the Borel subgroup $B$ of $G$ defined in Lemma~\ref{Borel}.
	
	Note that $(G_0)_\omega$ does not contain any element of order $9$, $K_0\cap Q\leqslant Q'$ by combining Lemma~\ref{lem:spec} and Lemma~\ref{uni} \rm(ii).
	Therefore, 
	\begin{equation*}
		|K_0|\leqslant q^2(q-1) <q(q^2-1) = |(G_0)_\omega|,
	\end{equation*}
   a contradiction to $|K_0|\geqslant |(G_0)_\omega|.$
   This completes the proof.
	\end{proof}

\begin{lem}\label{f:nq+1/4}
	If $G_\omega=N_{G}(A_1)$ and $K$ is a non-stable fixer of $G$ with $|K|\geqslant |G_\omega|,$ then $q=27$ and $K$ is the normalizer of some Sylow $2$-subgroup of $G.$
	Further, if $q=27$ and $G=G_0,$ then the normalizer of any Sylow $2$-subgroup is a fixer of $G.$  
\end{lem}
\begin{proof}
	Let $K_0$ be a fixer of $G_0$ such that $|K_0|\geqslant |(G_0)_\omega|.$
	
	Assume first $K_0$ is insoluble. 
	Note that $(G_0)_\omega$ does not contain any element of order $9,$ so $K_0$ does not contain any element of order $9$ by Lemma~\ref{lem:spec}.
	Thus, $K_0$ is isomorphic to one of $\PSL_2(q')$ and $K_0\cong \CC_2 \times \PSL_2(q')$ by Lemma~\ref{nonsol}.
	But then $K_0$ contains a cyclic subgroup of order $\frac{q'-1}{2}$ and $(G_0)_\omega$ does not contain any element of such order, a contradiction to Lemma~\ref{lem:spec}. 
	
	Thus, $K_0$ is soluble.
	Then $K_0$ is conjugate to a subgroup of one of the five soluble groups that are not $(G_0)_\omega=N_{G_0}(A_1)$ listed in Lemma~\ref{solgps}, ($K_0$ is not conjugate to $N_{G_0}(A_1)$ since $K$ is non-stable).
	We next deal with these five cases.
	Assume first $K_0$ is contained in a Borel subgroup of $G_0.$
	Then we may replace $K_0$ with $K_0^g$ for a suitable $g\in G_0$ and assume further $K_0$ is contained in the Borel subgroup $B=N_{G_0}(Q)$ defined in Lemma~\ref{Borel} by Lemma~\ref{p:fixers} \rm(iv). 
	Note that $(G_0)_\omega$ does not contain any element of order $9,$ $K_0$ does not contain any element of order $9$ by Lemma~\ref{lem:spec}, and hence $K_0\cap Q \leqslant Q'.$
	Since $|(G_0)_\omega|_3=3$ and $(G_0)_\omega$ contains an element of order $6,$ all elements of order $3$ in $(G_0)_\omega$ are conjugate to some elements in $Q'\backslash Z(Q)$ by Lemma~\ref{centuni}, and hence $K_0 \cap Z(Q)=1.$
	Thus, 
	$$K_0\cap Q \cong K_0\cap Q/K_0 \cap Z(Q)\leqslant Q'/Z(Q),$$
	and hence $|K_0\cap Q|<q.$
	On the other hand, $|(G_0)_\omega|$ is coprime to $\frac{q-1}{2},$ yielding $K_0$ does not contain any element of order dividing $\frac{q-1}{2}.$
	Therefore, 
	\begin{equation*}
		|K_0|\leqslant 2q< |(G_0)_\omega|,
	\end{equation*} 
	a contradiction.
	
	Assume next $K_0$ is conjugate to a subgroup of $N_{G_0}(A_2).$
	Note that $|(G_0)_\omega|$ is coprime to $|A_2|,$ $K_0$ does not contain any element of order dividing $|A_2|$ by Lemma~\ref{lem:spec}, yielding 
	\begin{equation*}
		|K_0|\leqslant 6 <|(G_0)_\omega|,
	\end{equation*} 
   a contradiction.
   A similar argument yields $K_0$ is not conjugate to a subgroup of $N_{G_0}(A_0)$ or $N_{G_0}(A_3).$
   
   It remains to consider $K_0$ is contained in a Sylow $2$-normalizer.
   Since 
   $$168\geqslant |K_0|> |(G_0)_\omega|=168,$$ it follows that $q=27$ and $K_0$ is the normalizer of a Sylow $2$-subgroup.
   
   Finally, assume $q=27$ and let $K_0$ be the normalizer of a Sylow $2$-subgroup, we prove that $K_0$ is a fixer of $G_0.$
   Note that $K_0\cong \mathrm{A\Gamma L}_1(8)$ by Lemma~\ref{normsyl2}.
   Let $x\in K_0,$ then the order of $x$ is one of $1,2,3,6,7$. 
   If $x$ is an involution, then $x$ is conjugate to an involution in $(G_0)_\omega,$ by noting that all involutions are conjugate in $G_0.$
   Assume $|x|=6.$ Then there exists $g\in G_0$ such that $x^g\in \la X_S(0,1,0)h(-1) \ra$ combing Lemma~\ref{Borel} and Lemma~\ref{centuni}.
   Similarly, every cyclic subgroup of order $6$ is conjugate to $\la X_S(0,1,0)h(-1) \ra.$
   Thus, $x$ is conjugate to an element in $G_\omega$ in $G_0.$
   By a similar argument, each element of order $3$ in $K_0$ is conjugate to an element of $(G_0)_\omega.$
   If $|x|=7,$ then $x$ is conjugate to some element in $(G_0)_\omega,$ by noting that the Sylow $7$-subgroup of $G_0$ is cyclic of order $7.$
   This completes the proof.   
	 
\end{proof}

\begin{lem}\label{f:a2a3}
	If $G_\omega=N_{G}(A_i),$ where $i\in \{2,3\},$ then $G$ does not have a non-stable fixer $K$ with $|K|\geqslant |G_\omega|.$
\end{lem}
\begin{proof}
	Assume $K$ is a non-stable fixer of $G$ and $|K|\geqslant |G_\omega|.$
	We first claim $K$ is soluble. 
	Assume the contrary $K$ is insoluble. 
	Then $K_0$ is isomorphic to one of the subgroup listed in Lemma~\ref{nonsol}.
	Thus, either $K_0$ contains an element of order $9,$ or $K_0$ contains a cyclic subgroup of order $\frac{q'-1}{2}.$
	On the other hand, $(G_0)_\omega$ does not contain an element of such order, a contradiction to Lemma~\ref{lem:spec}.  
	This completes the claim.
	
	Thus, $K_0$ is conjugate to a subgroup of the six soluble groups listed in Lemma~\ref{solgps}.
	We next deal with these six cases.
	Assume first $K_0$ a subgroup of Borel subgroup.
	In view of Lemma~\ref{p:fixers} \rm(iv), we may replace $K$ with $K^g$ for a suitable $g\in G_0$ and assume $K_0\leqslant B,$ where $B$ is the Borel subgroup defined in Lemma~\ref{Borel}.
	Note that $|(G_0)_\omega|=3,$ and $(G_0)_\omega$ contains an element of order $6,$ each element of order $3$ in $(G_0)_\omega$ is conjugate to an element in $Q'\backslash Z(Q)$ by Lemma~\ref{centuni}.
	Thus, $K_0\cap Q\leqslant Q'$ and $K_0\cap Z(Q)=1,$ yielding 
	\begin{equation*}
		|K_0\cap Q|=|K_0\cap Q'/K_0\cap Z(Q)|\leqslant |Q'/Z(Q)|=q.
	\end{equation*}
    On the other hand, since $|(G_0)_\omega|$ is coprime to $\frac{q-1}{2},$ it follows that $K_0$ does not contain any element of order dividing $\frac{q-1}{2}.$
    Thus, 
    \begin{equation*}
    	|K_0|\leqslant 2q \leqslant |(G_0)_\omega|,
    \end{equation*}
   a contradiction.
   
   If $K_0$ is conjugate to a subgroup of $N_{G_0}(A_0),$ then $K_0$ does not contain any element of order dividing $\frac{q-1}{2}$ by the above argument, and hence $|K_0|\leqslant 4<|(G_0)_\omega|,$ a contradiction.
   
   We next assume $K_0$ is conjugate to a subgroup of $N_{G_0}(A_1).$
   Note that $(G_0)_\omega$ does not contain any element of order dividing $|A_1|.$
   It follows that $|K_0|\leqslant 24<|(G_0)_\omega|,$ a contradiction.
   
   By a similar argument of the previous paragraph, $K_0$ is not conjugate to a subgroup of $N_{G_0}(A_j)$ for $j\in \{2,3\} \backslash i.$
   
   Finally, assume $K_0$ is contained in a normalizer of a Sylow $2$-subgroup of $G_0.$ 
   Then $|(G_0)_\omega| \leqslant |K_0|\leqslant 168 ,$ and hence $i=3$ and $q=27.$
   Thus, $(G_0)_\omega$ does not contain any element of order $7,$ and hence $K_0$ does not contain any element of order $7$ by Lemma~\ref{lem:spec}, yielding 
   $$|K_0|\leqslant 24 < |(G_0)_\omega|,$$ 
    a contradiction.
    
   This completes the proof.
	\end{proof}

\begin{lem}\label{f:subfi}
	Assume $G$ is a maximal subfield subgroup such that $G_0\cong {}^2G_2(q_0),$ and $K$ is a non-stable fixer with $|K|\geqslant |G_\omega|.$
	Then $K$ is a subgroup of the normalizer of some Sylow $3$-subgroup of $G.$
\end{lem}
\begin{proof}
	We first claim that $K$ is soluble.
	Assume the contrary that $K$ is insoluble.
	Then $K_0$ is isomorphic to one of the four subgroups listed in Lemma~\ref{nonsol}.
	Note that $K_0$ is a fixer of $G_0$ with $|K_0|\geqslant |(G_0)_\omega|$ by Lemma~\ref{lem:redu}.
	It follows that $K_0$ is isomorphic to one of $\PSL_2(q'), \CC_2 \times \PSL_2(q')$ and ${}^2G_2(q')$ with $q'\geqslant q_0.$
	Thus, $K_0$ contains a cyclic subgroup of order $\frac{q'-1}{2},$ but $(G_0)_\omega$ does not contain any element of such order, a contradiction.
	This completes the claim.
	
	Thus, $K_0$ is conjugate to a subgroup of the six groups listed in Lemma~\ref{solgps}.
	We next deal with these six cases, and conclude $K_0$ is conjugate to a subgroup of a maximal Borel subgroup.
	Assume first $K_0$ is conjugate to a subgroup of $N_{G_0}(A_0).$
	Note that $\gcd(|(G_0)_\omega|, |A_0|)=\frac{q_0-1}{2},$ $K_0$ is conjugate to a subgroup of $D_{q_0-1}$ by Lemma~\ref{lem:spec}, and hence $|K_0|< |(G_0)_\omega|$, a contradiction.
	
	Assume next $K_0$ is conjugate to a subgroup of $N_{G_0}(A_1).$
	If $q=q_0^3,$ then 
	\begin{equation*}
		|K_0|=2^33\times \frac{q+1}{4}< 2^3q_0^3\frac{q_0-1}{2}\frac{q+1}{4}=|(G_0)_\omega|,
	\end{equation*}
	a contradiction.
	Thus, $q\neq q_0^3,$ and hence $\gcd(|A_1|,|(G_0)_\omega|)=\frac{q_0+1}{4}.$
	Therefore,
	\begin{equation*}
		|K_0|\leqslant 6(q_0+1) < |(G_0)_\omega|,
	\end{equation*}
	a contradiction.
	
	Assume $K_0$ is conjugate to a subgroup of $N_G(A_i),$ $i\in \{2,3\}.$
	Then $\gcd(|A_i|,|(G_0)_\omega|)$ is one of $q_0+\sqrt{3q_0}+1$ and  $q_0-\sqrt{3q_0}+1,$ and hence
	\begin{equation*}
		|K_0|\leqslant 6(q_0+\sqrt{3q_0}+1)<|(G_0)_\omega|.
	\end{equation*} 
	
	If $K_0$ is conjugate to a subgroup of the normalizer of a Sylow $2$-subgroup, then $|K_0|=168<|(G_0)_\omega|,$ a contradiction.
	
	Therefore, $K_0$ is conjugate to a subgroup of a Borel subgroup, this completes the proof.
\end{proof}

\subsection*{Proof of Lemma~\ref{lem:reds}}
Lemma~\ref{lem:reds} holds by Lemma~\ref{f:borel}$\sim$\ref{f:subfi}.

\section{Proof of Theorem~\ref{thm:main}}\label{sec:proof}
In this section, we further deal with the two cases in Lemma~\ref{lem:reds}.
We first analyze the $G_0$-classes of elements of order $9$ in $(G_0)_\omega \cong {}^2G_2(q_0)$ in Lemma~\ref{subfie_uni}, and eliminate the case $r=3.$
Then we involve in the field automorphisms of $G_0\cong {}^2G_2(q)$ and prove Theorem~\ref{thm:main}.

\begin{lem}\label{subfie_uni}
	Let $q=q_0^r,$ where $r$ is an odd prime, and let $(G_0)_\omega=C_{G_0}(\phi^r).$
	If $3$ is coprime to $[\mathbb{F}_q : \mathbb{F}_{q_0}],$ then all elements of order $9$ is $G_0$-conjugate to some element in $(G_0)_\omega$, if $3$ divides $[\mathbb{F}_q : \mathbb{F}_{q_0}],$ then $(G_0)_\omega$ intersects non-trivially with a unique $G_0$-class of elements of order $9$. 
	In particular, if $3$ divides $[\mathbb{F}_q:\mathbb{F}_{q_0}],$ then each element of order $9$ in $G_0$ is conjugate to its inverse in $G_0$.
\end{lem}
\begin{proof}
	Assume first $3$ is coprime to $[\mathbb{F}_q :\mathbb{F}_{q_0}]$.
	We take $\lambda\in \mathbb{F}_{q_0}$ such that $\mathrm{Tr}_{\mathbb{F}_{q_0}/\mathbb{F}_3}(\lambda)\neq 0,$ it follows that 
	\begin{equation*}
		\mathrm{Tr}_{\mathbb{F}_q/\mathbb{F}_3}(\lambda)=\mathrm{Tr}_{\mathbb{F}_{q_0}/\mathbb{F}_3}(\mathrm{Tr}_{\mathbb{F}_q/\mathbb{F}_{q_0}}(\lambda))=[\mathbb{F}_q : \mathbb{F}_{q_0}]\mathrm{Tr}_{\mathbb{F}_{q_0}/\mathbb{F}_3}(\lambda)\neq 0.
	\end{equation*}	
	Thus, $\lambda\neq t'-(t')^{3\theta}$ for any $t'\in \mathbb{F}_q$ by additive form of Hilbert 90, and hence $x_1=X_S(1, \lambda+1, 0)$ is not conjugate to $x_1^{-1}$ in $G_0={}^2G_2(q)$ by Lemma~\ref{conj9}.
	Therefore, $x_1,x_1^{-1}$ and $x_2=X_S(1,1,0)$ are representatives for three $G$-classes of elements of order $9$ by Lemma~\ref{conj9}, and they are all in $(G_0)_\omega.$
	
	Assume next $3$ divides $[\mathbb{F}_q:\mathbb{F}_{q_0}].$
	Let $x$ be an element of order $9$ in $(G_0)_\omega,$ we claim that $x$ is conjugate to $x^{-1}$ in $G.$
	To see this, note that $x=X_S(t,u,v)$ for some $t\in \mathbb{F}_{q_0}^\times$ and $u,v\in \mathbb{F}_{q_0}.$
	Also note that for any $\lambda\in \mathbb{F}_{q_0},$ 
	\begin{equation*}
		\mathrm{Tr}_{\mathbb{F}_q/\mathbb{F}_3}(\lambda)=\mathrm{Tr}_{\mathbb{F}_{q_0}/\mathbb{F}_3}(\mathrm{Tr}_{\mathbb{F}_q/\mathbb{F}_{q_0}}(\lambda))=\mathrm{Tr}_{\mathbb{F}_q/\mathbb{F}_{q_0}}(0)=0.
	\end{equation*}
	Thus, $\mathrm{Tr}_{\mathbb{F}_q/\mathbb{F}_3}(\frac{u}{t^{1+\theta}}-1)=0,$ and hence there exists $t'\in \mathbb{F}_{q_0}$ such that $\frac{u}{t^{3\theta+1}}-1=\frac{t'}{t}-(\frac{t'}{t})^{3\theta}.$
	Therefore, $u-t^{3\theta+1}=t^{3\theta}t'-(t')^{3\theta}t,$ and $x$ is conjugate to $x^{-1}$ by Lemma~\ref{conj9}.
	This completes the proof.
\end{proof}
\begin{lem}\label{subfield}
	Assume $G_\omega$ is a maximal subfield subgroup of $G$ such that $(G_0)_\omega \cong {}^2G_2(q_0)$, where $q=q_0^r$ for some prime $r$.
    Let $E_0=C_{Z(Q)}(\phi^r),$ where $Q$ is the unipotent subgroup of $G_0$ defined in Lemma~\ref{Borel}.
	Then $r\neq 3$, and $K_0$ is conjugate to a subgroup of $N_{G_0}(E_0)=Q{:}H(q_0),$ where  
	\begin{equation*}
		H(q_0)=\{h(s),s\in \mathbb{F}_{q_0}^\times\},
	\end{equation*}
   and $h(s)$ is the element defined in Lemma~\ref{Borel}.
	If further $G=G_0,$ then $N_{G_0}(E_0)$ is a fixer of $G_0.$ 
\end{lem}
\begin{proof}
	By Lemma~\ref{f:subfi}, we may assume $K_0$ is conjugate to a subgroup of $B=Q{:}H$ defined in \ref{Borel}.
	In view of Lemma~\ref{p:fixers} \rm(iv), we may replace $K$ with $K^g$ and $G_\omega$ with $G_\omega^{g'}$ for some suitable $g, g'\in G,$ and assume $K_0\leqslant B=Q{:}H,$ and $G_\omega=C_G(\phi^r).$ 
	We note that under the $7$-dimensional representation $\rho$ over $\mathbb{F}_q,$ the elements in $\rho(B)$ that are conjugate to some matrices with each entries in $\mathbb{F}_{q_0}$ in $\GL_7(q)$ coincides with those upper-triangular matrices with diagonal entries in $\mathbb{F}_{q_0},$ which is $Q{:}H(q_0)=\GL_7(q_0)\cap \rho(B).$
	Thus, $K_0\leqslant Q{:}H(q_0)=N_{G_0}(E_0),$ by noting that $(G_0)_\omega\leqslant \GL_7(q_0).$
	
	Assume $3\neq [\mathbb{F}_q:\mathbb{F}_{q_0}]$. Then combining Lemma~\ref{conj_uni} and \ref{subfie_uni}, $(G_0)_\omega$ contains all $G_0$-classes of unipotent elements.
	On the other hand, each element in $N_{G_0}(E_0)=Q{:}H(q_0)$ is either conjugate to $Q$ in $G_0$ or conjugate to $H(q_0)\leqslant (G_0)_\omega$ in $G_0$ by Hall's theorem.
	Thus, each element in $N_{G_0}(E_0)$ is conjugate to an element in $(G_0)_\omega,$ and hence $N_{G_0}(E_0)$ is indeed a fixer of $G_0.$ 
	
	Finally, we consider the case $3=[\mathbb{F}_q:\mathbb{F}_{q_0}].$
	Since each element of $K_0$ is conjugate to some element in $(G_0)_\omega,$ it follows that  each element of order $9$ in $(G_0)_\omega$ is conjugate to its inverse by Lemma~\ref{subfie_uni}.
	We next prove $|K\cap Q|\leqslant q^2.$ 
	To see this, let $x=X_S(t,u,v)\in K\cap Q,$ we first claim that $u\in t^{3\theta+1}A,$ where $A$ is defined in Section~\ref{p:field}. 
	Note that $x$ is conjugate to $x^{-1},$ and hence there exists $t'\in \mathbb{F}_q$ such that $u-t^{3\theta+1}=-t(t')^{3\theta}+t't^{3\theta}$ by Lemma~\ref{conj9}.
	Since
	\begin{equation*}
		\{-t(t')^{3\theta}+t't^{3\theta} \mid t'\in \mathbb{F}_q\}=\{t^{3\theta+1}(-(\frac{t'}{t})^{3\theta}+\frac{t'}{t})\mid t'\in \mathbb{F}_q\}=t^{3\theta+1}A,
	\end{equation*}
	it follows that $u-t^{3\theta+1}\in t^{3\theta+1}A.$
	Thus, it remains to show that $t^{3\theta+1}\in t^{3\theta+1}A,$ which is equivalent to $1\in A.$
	To see this, note that $$\mathrm{Tr}_{\mathbb{F}_q/\mathbb{F}_3}(1)=2n+1=0,$$
	applying additive form of Hilbert 90 there exists $t'\in \mathbb{F}_q$ such that $1=-(t')^{3\theta}+t',$ yielding $1\in A.$ 
	Therefore, both $t^{3\theta+1}$ and $u-t^{3\theta+1}$ are in $t^{3\theta+1}A,$ and hence $u\in t^{3\theta+1}A,$ this proves the claim.
	Note that the image of $X_S(t,u,v)Q'$ in $Q/Q'$ is determined by $t$,
    and hence $K\cap Q/(K\cap Q')$ can be identified with the following additive subgroup of $\mathbb{F}_q$ 
    \begin{equation*}
    \{t\in \mathbb{F}_q \mid \text{ there exists some $u,v \in \mathbb{F}_q$ such that $X_S(t,u,v)\in K$}\}.    
    \end{equation*}
	Similarly, $K\cap Q'/(K\cap Z(Q))$ can be identified with the additive subgroup 
    $$\{u\in \mathbb{F}_q\mid \text{there exists $v\in \mathbb{F}_q$ such that $X_S(0,u,v)\in K$} \}.$$

	We next prove that 
	\begin{equation}\label{e:dim}
		\mathrm{dim}_{\mathbb{F}_3}(K\cap Q/(K\cap Q'))+\mathrm{dim}_{\mathbb{F}_3}(K\cap Q'/K\cap Z(Q))\leqslant 2n+1.
	\end{equation}
	To see this, note that
	for each $u\in K\cap Q'/(K\cap Z(Q))$ and each $t\in K\cap Q/(K\cap Q'),$ there exists some $u_1,v_1$ and $v_2$ in $\mathbb{F}_q$ such that $X_S(t,u_1,v_1)\in K$ and $X_S(0,u,v_2)\in K.$
	Thus,  
	\begin{equation*}
		X_S(t,u_1,v_1)X_S(0,u,v_2)^{-1}=X_S(t,u_1-u,v_1-v_2)\in K\cap Q,
	\end{equation*}
	and hence both $u_1$ and $u_1-u$ are in $t^{3\theta+1}A,$ yielding $u\in t^{3\theta+1}A.$
	As a result, 
	\begin{equation*}
		u\in \bigcap_{t\in K\cap Q/K\cap Q'} t^{3\theta+1}A,
	\end{equation*}
	yielding $$\mathrm{dim}_{\mathbb{F}_3}(K\cap Q'/(K\cap Z(Q)))\leqslant \mathrm{dim}_{\mathbb{F}_3}(\bigcap_{t\in K\cap Q/K\cap Q'} t^{3\theta+1}A).$$
	On the other hand, Lemma~\ref{lem:dimintA} implies that 
	$$\mathrm{dim}_{\mathbb{F}_3}(\bigcap\limits_{t\in K\cap Q/K\cap Q'} t^{3\theta+1}A)=2n+1-\mathrm{dim}_{\mathbb{F}_3}(\mathrm{span}\{t^{-(3\theta+1)}\mid t\in K\cap Q/(K\cap Q')\}).$$
	As a result, to verify (\ref{e:dim}), it suffices to show that
	\begin{equation*}
		\mathrm{dim}_{\mathbb{F}_3}(K\cap Q/(K\cap Q'))\leqslant \mathrm{dim}_{\mathbb{F}_3}(\mathrm{span}\{t^{-(3\theta+1)}\mid t\in K\cap Q/(K\cap Q')\}),
	\end{equation*}
	which is simply by noting that 
	\begin{equation*}
		\begin{aligned}
			|\mathrm{span}_{\mathbb{F}_3}\{t^{-(3\theta+1)}\mid t\in K\cap Q/(K\cap Q')\}|&\geqslant |\{t^{-(3\theta+1)}\mid t\in K\cap Q/(K\cap Q')\}|\\
			&>\frac{1}{2}|K\cap Q/(K\cap Q')|.
		\end{aligned}
	\end{equation*} 
	Therefore, (\ref{e:dim}) holds, and hence $|K\cap Q|\leqslant q^2.$
	Thus, $|K_0|\leqslant q^2q_0<|(G_0)_\omega|,$ which contradicts to Lemma~\ref{lem:redu}.
	
	This completes the proof.
\end{proof}

Now we consider the case $G_0\trianglelefteqslant G \trianglelefteqslant \Aut(G_0),$ and $G\neq G_0.$
Let $\phi\in \Aut(G_0)$ be the Frobenius automorphism which sends each Chevalley element $x_r(t)$ to $x_r(t^3),$ where $r\in \Phi$ and $t\in \mathbb{F}_q.$ 
Then $\Aut(G_0)=G_0{:}\la \phi \ra,$ and $G=G_0{:}\la \psi \ra$ for some $\psi \in \la \phi \ra.$

We next deal with Case \rm(i) of Lemma~\ref{lem:reds} by Lemma~\ref{lem:subfield1} and \ref{lem:subfield2}, we deal with Case \rm(ii) of Lemma~\ref{lem:reds} by Lemma~\ref{lem:q27}, then we finish the proof of Theorem~\ref{thm:main}.
\begin{lem}\label{lem:subfield1}
	Assume $G_\omega$ is a maximal subfield subgroup of $G,$ such that $(G_0)_\omega\cong {}^2G_2(q_0),$ where $q=q_0^r$ and $r\neq 3$ is an odd prime.
	Then $N_{G}(E_0)=(Q{:}H(q_0)){:}\la \psi \ra$ is a fixer of $G,$ where $Q$ is a unipotent subgroup of $G_0$ defined in Lemma~\ref{Borel}, $E_0=C_{Z(Q)}(\phi^r),$ and 
	\begin{equation*}
		H(q_0)=\{h(s)\mid s\in \mathbb{F}_{q_0}^\times\}
	\end{equation*} 
is a subgroup of $H$ defined in Lemma~\ref{Borel}.
\end{lem}
\begin{proof}
	In view of Lemma~\ref{lem:redu}, we may assume $G=\Aut(G_0)=G_0{:}\la \phi \ra,$ and hence $\psi=\phi.$ 
	Up to a conjugation in $G,$ we may assume $(G_0)_\omega=C_{G_0}(\phi^r)$ and $G_\omega=C_{G}(\phi^r)$.
	Let $K=N_G(E_0)=(Q{:}H(q_0)){:}\la \phi \ra\leqslant G,$ and let $x$ be an arbitrary element in $K_0$, we aim to prove that $x$ is conjugate to an element in $G_\omega$ in $G.$
	By Lemma~\ref{uni}, one of the following four cases hold: 
	\rm(i) $|\la x \ra \cap Q|=3$ and $\la x \ra\cap Q\leqslant Z(Q)$;
	\rm(ii) $\la x \ra \cap Q=1$;
	\rm(iii) $|\la x \ra\cap Q|=3$ and $\la x \ra\cap Z(Q)=1$;
	\rm(iv) $|\la x \ra \cap Q|=9.$
	We then deal with these four cases one by one. 
	
	\subsubsection*{Case \rm(i) $|\la x \ra \cap Q|=3$ and $\la x \ra\cap Q\leqslant Z(Q)$ }\
	 
	Note that there exists $h\in H$ such that $\la x^h \ra \cap Q=\la X_S(0,0,1) \ra$ by Lemma~\ref{Borel} and Lemma~\ref{uni}, it follows that $$x^h\in C_{N_G(Q)}(X_S(0,0,1))=Q{:}\la \phi \ra.$$ 
	So there exists $t, u, v\in \mathbb{F}_q$ and some integer $j$ such that $x^h=X_S(t,u,v)\phi^j.$
	We next claim that there exists $g\in Q$ such that $x^{hg}\in Z(Q){:}\la \phi \ra$.
	To see this, let $k=|\phi^j|,$ then 
	\begin{equation*}
		(x^h)^k=X_S(t+t^{\phi^j}+\dots+ t^{\phi^{(k-1)j}}, u_1, v_1)=X_S(0,0,v_1), 
	\end{equation*}
	where $v_1\in \mathbb{F}_3^\times.$
	Thus, $t+t^{\phi^j}+\dots+ t^{\phi^{(k-1)j}}=0,$ and hence there exists $t_1\in \mathbb{F}_q$ such that $t=t_1^{\phi^{-j}}-t_1$ by additive form of Hilbert 90, yielding 
	\begin{equation*}
		x^{hX_S(-t_1,0,0)}=(X_S(t,u,v)\phi^j)^{X_S(-t_1,0,0)}=X_S(t_1+t-t_1^{\phi^{-j}},u_0,v_0)\phi^j=X_S(0,u_0,v_0)\phi^j,
	\end{equation*}
    for some $u_0,v_0 \in \mathbb{F}_q,$ and hence $x^{hX_S(-t_1,0,0)}\in Q'\phi^j.$
	By a similar argument, there exists $u_1 \in \mathbb{F}_q$ such that $x^{hX_S(-t_1,0,0)X_S(0,u_1,0)}\in Z(Q)\phi^j.$
	Thus, the claim holds by taking $g$ to be $X_S(t_1,0,0)X_S(0,u_1,0).$
	So there exists $v_1\in \mathbb{F}_q$ such that $x^{hg}=X_S(0,0,v_1)\phi^j.$
	Let $q=q_1^k,$ then $\mathbb{F}_{q_1}$ is the fixed field of $\la \phi^j \ra$ on $\mathbb{F}_q.$
	Since \begin{equation*}
	(x^{hg})^k=X_S(0,0,v_1+v_1^{\phi^j}+\dots+v_1^{\phi^{(k-1)j}})=X_S(0,0,\mathrm{Tr}_{\mathbb{F}_q/\mathbb{F}_{q_1}}(v_1))\in \la x^{hg} \ra\cap Q
	\end{equation*}
    and $|\la x^{hg} \ra\cap Q|=3,$
    it follows that $\mathrm{Tr}_{\mathbb{F}_q/\mathbb{F}_{q_1}}(v_1)\in \mathbb{F}_{q_1}^\times.$
	Let $\lambda\in \mathbb{F}_{q_0}$ such that $\mathrm{Tr}_{\mathbb{F}_q/\mathbb{F}_{q_1}}(\lambda)\in \mathbb{F}_{q_1}^\times,$ ($\lambda$ exists by Lemma~\ref{trqq1}), and let $\mu=\mathrm{Tr}_{\mathbb{F}_q/\mathbb{F}_{q_1}}(v_1)\mathrm{Tr}_{\mathbb{F}_q/\mathbb{F}_{q_1}}(\lambda)^{-1}.$ 
	Then $\mu \in \mathbb{F}_{q_1}^\times,$ and
	$\mathrm{Tr}_{\mathbb{F}_q/\mathbb{F}_{q_1}}(\mu^{-1}v_1)=\mathrm{Tr}_{\mathbb{F}_q/\mathbb{F}_{q_1}}(\lambda),$ and hence there exists $\lambda_1\in \mathbb{F}_q$ such that $\lambda-\mu^{-1}v_1=\lambda_1^{\phi^{-j}}-\lambda_1$ by additive form of Hilbert 90.
	Thus, by a direct calculation with the algorithm in Lemma~\ref{Borel} and \ref{uni},
	\begin{equation*}
		\begin{aligned}
		(x^{hg})^{h(\mu)X_S(0,0,\lambda_1)}&=(X_S(0,0,\mu^{-1}v_1)\phi^j)^{X_S(0,0,\lambda_1)}\\
		&=X_S(0,0,\mu^{-1}v_1+\lambda_1^{\phi^{-j}}-\lambda_1)\phi^j
		&=X_S(0,0,\lambda)\phi^j\in G_\omega.
		\end{aligned}
	\end{equation*}
	
	\subsubsection*{Case \rm(ii) $\la x \ra \cap Q=1$}\
	
	Let $x=x_3x_3'=x_3'x_3,$ where $|x_3|$ is a power of $3$ and $|x_{3'}|$ is coprime to $3.$ ($x_3$ and $x_3'$ exists by \cite[Lemma 22.18]{JL_repbook}).
    Similarly, let $\phi=\phi_3\phi_{3'}=\phi_{3'}\phi_3,$ with $|\phi_3|$ a power of $3$ and $|\phi_{3'}|$ coprime to $3.$
	Note that there exists $g_0\in K$ such that $x_3^{g_0}\in Q{:}\la \phi_3 \ra,$ since $Q{:}\la \phi_3 \ra$ is a Sylow-$3$ subgroup of $K.$
	By a similar argument in the case when $\la x \ra \cap Q$ is a subgroup of $Z(Q)$ of order $3,$ there exists $g_1\in Q$ such that $x_3^{g_0g_1}\in \la \phi \ra$.
	Thus, there exists  some integer $i$ such that $x_3^{g_0g_1}=\phi^i,$ and $|\phi^i|$ is a $3$-power. 
	Since $g_0g_1\in K$ and $$(x_{3'})^{g_0g_1}\in C_K(x_3^{g_0g_1})=C_K(\phi^i)=(C_{\phi^i}(Q){:}C_{\phi^i}(H(q_0))){:}\la \phi \ra,$$
	 it follows that there exists $g_2\in C_K(\phi_3^i)$ such that $x_{3'}^{g_0g_1g_2}$ is contained in a Hall $3'$-subgroup $C_{\phi_3^i}(H(q_0)){:}\la \phi_{3'} \ra$ of  $C_K(\phi_3^i).$
	 Thus, 
	 \begin{equation*}
	 	\la x^{g_0g_1g_2} \ra \leqslant \la x_3^{g_0g_1g_2} \ra \la x_{3'}^{g_0g_1g_2} \ra \leqslant C_{\phi_3^i}(H(q_0)){:}\la \phi \ra\leqslant G_\omega.
	 \end{equation*}
    This finishes the case $\la x \ra \cap  Q\leqslant Z(Q).$
    
    \subsubsection*{Case \rm(iii) $|\la x \ra \cap Q|=3$ and $\la x \ra\cap Z(Q)=1$ }\

    Note that there exists $g\in G$ such that $\la x^g \ra \cap Q'=\la X_S(0,1,0) \ra$ by Lemma~\ref{conj_uni}, and hence 
    \begin{equation*}
    	x^g \in C_{G}(X_S(0,1,0))=(Q'{:}\la h(-1) \ra){:}\la \phi \ra
    \end{equation*}
   by Lemma~\ref{uni}.
   Let $x^g=x_{2'}^gx_2^g=x_2^gx_{2'}^g,$ where $|x_{2'}|$ is odd and $|x_2|$ is a $2$-power ($x_2$ and $x_{2'}$ exists by \cite[Lemma 22.18]{JL_repbook}).
   Then $x_{2'}^g\in Q'{:}\la \phi \ra,$ since $Q'{:}\la \phi \ra$ is the unique Hall $2'$-subgroup of $(Q'{:}\la h(-1) \ra){:}\la \phi \ra.$
   Thus, there exists $u_2,v_2\in \mathbb{F}_q$ and some integer $i$ such that $x_{2'}^g=X_S(0,u_2,v_2)\phi^i.$
   Let $k=|\phi^i|$ and $q=q_1^k$, so $\mathbb{F}_{q_1}$ is the fixed field of $\phi^i$. 
   Note that 
   \begin{equation*}
   	\begin{aligned}
   	(x_{2'}^g)^k&=X_S(0,u_2+u_2^{\phi^i}+\dots+u_2^{\phi^{(k-1)i}},v_2+v_2^{\phi^i}+\dots+v_2^{\phi^{(k-1)i}})\\
   	&=X_S(0,\mathrm{Tr}_{\mathbb{F}_q/\mathbb{F}_{q_1}}(u_2),\mathrm{Tr}_{\mathbb{F}_q/\mathbb{F}_{q_1}}(v_2))\in \la X_S(0,1,0) \ra,
   	\end{aligned}
   \end{equation*}
   it follows that $\mathrm{Tr}_{\mathbb{F}_q/\mathbb{F}_{q_1}}(u_2)\in \{\pm 1\},$ and there exists $\lambda \in \mathbb{F}_{q_0}$ such that $\mathrm{Tr}_{\mathbb{F}_q/\mathbb{F}_{q_1}}(\lambda)=\mathrm{Tr}_{\mathbb{F}_q/\mathbb{F}_{q_1}}(u_2)$ by Lemma~\ref{trqq1}.
   Thus, there exists $u_3$ and $v_3$ in $\mathbb{F}_q$ such that $\lambda-u_2=u_3^{\phi^{-i}}-u_3$ and $v_2=v_3^{\phi^{-i}}-v_3$ by additive form of Hilbert 90, yielding
   \begin{equation*}
   	\begin{aligned}
   	x_{2'}^{gX_S(0,u_3,-v_3)}&=(X_S(0,u_2,v_2)\phi^i)^{X_S(0,u_3,-v_3)}=X_S(0,u_2+u_3^{\phi^{-i}}-u_3,v_2-v_3^{\phi^{-i}}+v_3)\phi^i\\
   	&=X_S(0,\lambda,0)\phi^i \in G_\omega.
   	\end{aligned}
   \end{equation*}
  For convenience, let $g_1=X_S(0,u_3,-v_3)\in Q',$ we claim there exists $g_2\in C_G(x_{2'}^{gg_1})$ such that 
  $x_2^{gg_1g_2}\in \la h(-1) \ra.$
  To see this, first note that $C_G(x^{gg_1}_{2'}) \leqslant (Q'{:}\la h(-1) \ra){:}\la \phi \ra,$ this is because $\la x^{gg_1}_{2'} \ra \cap Q'=\la X_S(0,1,0) \ra$ and $$C_G(x^{gg_1}_{2'})\leqslant C_G(\la x^{gg_1}_{2'} \ra \cap Q')=C_G(\la X_S(0,1,0)\ra )=(Q'{:}\la h(-1) \ra){:}\la \phi \ra.$$
  On the other hand, $\la h(-1) \ra$ commutes with $x_{2'}^{gg_1}=X_S(0,\lambda, 0)\phi^i,$ it follows that $\la h(-1) \ra$ is a Sylow $2$-subgroup of $C_G(x_{2'}^{gg_1}).$
  Thus, there exists $g_2\in C_G(x_{2'}^{gg_1})$ such that $x_2^{gg_1g_2}\in \la h(-1) \ra,$ by noting that $x_2x_{2'}=x_{2'}x_2.$
  This proves the claim.
  Therefore, 
  $$x^{gg_1g_2}=x_{2'}^{gg_1g_2}x_2^{gg_1g_2}\in G_\omega,$$
  this completes the case $|\la x \ra \cap Q|=3$ and $\la x \ra\cap Z(Q)=1$.  
  
  \subsubsection*{Case \rm(iv) $|\la x \ra \cap Q|=9$ }\
  
  We claim that there exists $u_0\in \mathbb{F}_{q_0}$ such that $\la x \ra \cap Q$ is conjugate to $\la X_S(1,u_0,0) \ra$ in $G_0.$
  To see this, let $k$ be the order of $xG_0$ in $G/G_0$. Then $x^k$ is an element of order $9$ in $Q.$
  If $x^k$ is conjugate to $x^{-k}$ in $G_0$, then $x^k$ is conjugate to $X_S(1,1,0)$ in $G_0$ combining Lemma~\ref{conj9} and \ref{conj_uni}.
  Assume $x^k$ is not conjugate to $x^{-k}$ in $G_0.$ 
  Let $u_0\in \mathbb{F}_{q_0}$ be an element such that $\mathrm{Tr}_{\mathbb{F}_q/\mathbb{F}_3}(u_0-1)\neq 0,$ ($u_0$ exists by Lemma~\ref{trqq1}). 
  Then there does not exist $t'\in \mathbb{F}_q$ such that $u_0-1=-(t')^{3\theta}+t'$ by additive form of Hilbert 90, and hence $X_S(1,u_0,0)$ is not conjugate to $X_S(1,u_0,0)^{-1}$ in $G_0$ by Lemma~\ref{conj9}.
  Thus, $x^k$ is conjugate to one of $X_S(1,u_0,0)$ and $X_S(1,u_0,0)^{-1}$ in $G_0$ by Lemma~\ref{conj_uni}, and this completes the claim.
  Thus, there exists $g\in G$ such that $x^g\in C_G(X_S(1,u_0,0)).$
  We next claim that $x^g\in Q{:}\la \phi \ra.$ 
  To see this, first note that
  $$x^g\in C_G( X_S(1,u_0,0)) \leqslant N_G(Q)$$ by \cite[Lemma~1]{Levchuk85}.
  Further, in $N_G(Q)/Q'\cong \mathrm{A\Gamma L}_1(q),$ it is not difficult to see that 
  \begin{equation*}
  	C_{N_G(Q)/Q'}(X_S(1,u_0,0)Q')=(Q{:}\la \phi \ra)/Q'.
  \end{equation*} 
 This proves the claim.
 Thus, there exists $t_1,u_1,v_1\in \mathbb{F}_q$ and some integer $i$ such that $x^g=X_S(t_1,u_1,v_1)\phi^i.$
 Since $X_S(1,u_0,0)^{x^g}=X_S(1,u_0,0),$ 
 \begin{equation}\label{eqc}
 	\begin{aligned}
 	X_S(1,u_0,0)^{X_S(t_1,u_1,v_1)}&=X_S(1,u_0-(t_1)^{3\theta}+t_1,-t_1u_0+u_1+t_1^{3\theta+1}-t_1-t_1^{3\theta}+t_1^2)\\
 	&=X_S(1,u_0,0)^{\phi^{-i}}=X_S(1,u_0^{\phi^{-i}},0)
 	\end{aligned}
 \end{equation}
 by Lemma~\ref{uni} \rm(iii).
 Thus, $u_0^{\phi^{-i}}-u_0=-(t_1)^{3\theta}+t_1.$ 
 Note that $\mathrm{Tr}_{\mathbb{F}_{q_0}/\mathbb{F}_3}(u_0^{\phi^{-i}}-u_0)=0,$   and hence there exists $t'\in \mathbb{F}_{q_0}$ such that $u_0^{\phi^{-i}}-u_0=-(t')^{3\theta}+t'$ by additive form of Hilbert 90.
 It follows that $(t_1-t')^{3\theta}-(t_1-t')=0,$ and hence $t_1-t'\in \mathbb{F}_3,$ yielding $t_1\in \mathbb{F}_{q_0}.$
 Thus, $u_1=t_1u_0-t_1^{3\theta+1}+t_1+t_1^{3\theta}-t_1^2\in \mathbb{F}_{q_0}.$
 Recall that $k=|\phi^i|$ is the order of $x^gG_0$ in $G/G_0,$ let $q=q_1^k.$
 We finally find $g_1\in Q$ such that $x^{gg_1}\in G_\omega.$
 To do this, first note that
 \begin{equation*}
 	\begin{aligned}
 	(x^g)^k&=(X_S(t_1,u_1,v_1)\phi^i)^k=X_S(t_1+t_1^{\phi^i}+\dots+t_1^{\phi^{(k-1)i}},u',v')\\
 	&=X_S(\mathrm{Tr}_{\mathbb{F}_q/\mathbb{F}_{q_1}}(t_1),u',v')\in \la X_S(1,u_0,0) \ra,
 	\end{aligned}
 \end{equation*}
  and hence $\mathrm{Tr}_{\mathbb{F}_q/\mathbb{F}_{q_1}}(t_1)\in \{\pm 1\}.$
  Let $\mu=-\mathrm{Tr}_{\mathbb{F}_q/\mathbb{F}_{q_1}}(v_1)(\mathrm{Tr}_{\mathbb{F}_q/\mathbb{F}_{q_1}}(t_1))^{-1}.$
  Then $\mu\in \mathbb{F}_{q_1},$ and $\mathrm{Tr}_{\mathbb{F}_q/\mathbb{F}_{q_1}}(v_1+\mu t_1)=0,$ and hence there exists $v_2\in \mathbb{F}_q$ such that $v_1+\mu t_1=v_2^{\phi^{-i}}-v_2$ by additive form of Hilbert 90.
  Therefore, 
  \begin{equation*}
  	\begin{aligned}
  	(x^g)^{X_S(0,\mu,0)X_S(0,0,v_2)}&=(X_S(t_1,u_1,v_1)\phi^i)^{X_S(0,\mu,0)X_S(0,0,-v_2)}\\
  	&=(X_S(t_1,u_1,v_1)^{X_S(0,\mu,0)}\phi^i)^{X_S(0,0,-v_2)}\\
  	&=(X_S(t_1,u_1,v_1+\mu t_1)\phi^i)^{X_S(0,0,-v_2)}\\
  	&=X_S(t_1,u_1,v_1+\mu t_1-(v_2^{\phi^{-i}}-v_2))\phi^i\\
  	&=X_S(t_1,u_1,0)\phi^i\in G_\omega.
  	\end{aligned}
  \end{equation*}
  This completes the proof.
 	\end{proof}

 \begin{lem}\label{lem:subfield2}
 	Assume $G_\omega$ is a maximal subfield subgroup such that $(G_0)_\omega\cong {}^2G_2(q_0),$ where $q=q_0^r,$ and $r\neq 3$ is an odd prime, and $K\leqslant G$ is a fixer.
 	Then $K$ is conjugate to a subgroup of $N_G(E_0)=(Q{:}H(q_0)){:}\la \phi \ra,$ where $Q$ is the unipotent subgroup of $G_0$ defined in \ref{Borel}, $E_0=C_{Z(Q)}(\phi^r),$ and
 	\begin{equation*}
 		H(q_0)=\{h(s)\mid s\in \mathbb{F}_{q_0}^\times\},
 	\end{equation*}  
 where $h(s)$ is the element defined in Lemma~\ref{Borel}. 
 \end{lem}
\begin{proof}
	Note that we may replace $K$ with $K^g$ for some suitable $g\in G$ and assume $K_0\leqslant N_{G_0}(E_0)=Q{:}H(q_0)$ by Lemma~\ref{p:fixers} \rm(iv) and \ref{subfield}. 
    Then $K\cap Q\neq 1$ since $|K_0|\geqslant |(G_0)_\omega|,$ and hence $$K\leqslant N_G(K\cap Q)\leqslant N_G(Q)$$ by \cite[Lemma 1]{Levchuk85}.
	Let $x\in K$ such that $K=K_0\la x \ra,$ ($x$ exists since $K/K_0$ is a subgroup of $G/G_0,$ and $G/G_0$ is cyclic).
	Note that $N_{G_0}(E_0)=Q{:}H(q_0)\trianglelefteqslant N_G(Q),$ it suffices to show that $x^g\in N_G(E_0)=(Q{:}H(q_0)){:}\la \phi \ra$ for some $g\in N_G(Q).$
	
	We first assume $\la x \ra \cap Q \neq 1.$
	Let $k$ be an integer such that $x^k$ is a non-identity element in $\la x \ra \cap Q.$
	By a similar argument in Case \rm(i), \rm(iii) and \rm(iv) where $\la x \ra \cap Q\neq 1$ in the proof of Lemma~\ref{lem:subfield1}.
	There exists $g\in G$ such that 
	\begin{equation*}
		x^g\in C_G(x^k)\leqslant (Q{:}\la h(-1) \ra){:}\la \phi \ra.
	\end{equation*}
	Moreover, note that $g$ normalizes the subgroup of $Q$ generated by $(x^k)^{\la g \ra},$ $g\in N_G(Q)$ by applying \cite[Lemma~1]{Levchuk85}.
	This completes the case $\la x \ra \cap Q\neq 1.$
	
	It remains to consider $\la x \ra \cap Q=1.$
	Note that $x\in N_G(Q)$, by a similar argument in Case \rm(ii) of the proof of Lemma~\ref{lem:subfield1}, there exists $g_0\in N_G(Q)$ such that $x^{g_0}\in H{:}\la \phi \ra.$
    Thus, $$\la x^{g_0} \ra \cap K_0 \leqslant H\cap K_0\leqslant H(q_0),$$ and one of the following three cases holds: \rm(i): $\la x \ra \cap K_0=1$, \rm(ii): $|\la x \ra \cap K_0|=2$, \rm(iii): $|\la x \ra \cap K_0|$ is a divisor of $q_0-1$ greater than $2$, by noting that $\la x^{g_0} \ra \cap K_0=(\la x \ra \cap K_0)^{g_0},$ we next analyze these three cases one by one.
	
	\subsubsection*{Case \rm(i): $\la x \ra \cap K_0=1$}
	Since $x^{g_0}\in H{:}\la \phi \ra$, there exists $s\in \mathbb{F}_{q}^\times$ and some integer $i$ such that $x^{g_0}=h(s)\phi^i.$
	Let $|\phi^i|=k$ and $q=q_1^k.$  
	Then
	\begin{equation*}
		(x^{g_0})^k=h(ss^{\phi^i}s^{\phi^{2i}}\dots s^{\phi^{(k-1)i}})=h(\mathrm{N_{\mathbb{F}_q/\mathbb{F}_{q_1}}(s)})\in \la x^{g_0} \ra \cap K_0=1=h(1),
	\end{equation*}
   where $\mathrm{N}_{\mathbb{F}_q/\mathbb{F}_{q_1}}(s)$ is the norm of $s$ of the extension $\mathbb{F}_{q}/\mathbb{F}_{q_1}.$
   Thus, $\mathrm{N_{\mathbb{F}_q/\mathbb{F}_{q_1}}(s)}=1,$ and hence there exists $\mu \in \mathbb{F}_q^\times$ such that $s=\mu^{\phi^{-i}}\mu^{-1}$ by Hilbert 90, yielding
   \begin{equation*}
   	(x^{g_0})^{h(\mu^{-1})}=h(\mu)h(s)\phi^ih(\mu^{-1})=h(\mu s\mu^{-\phi^{-i}})\phi^i=h(s(\mu^{\phi^{-i}}\mu^{-1})^{-1})\phi^i=\phi^i.
   \end{equation*}
   Therefore, by taking $g=g_0h(\mu^{-1})\in N_G(Q),$ $x^g=\phi^i,$ and hence $$K^g=K_0^g\la x^g \ra \leqslant (Q{:}H(q_0)){:}\la \phi \ra=N_G(E_0).$$
   This completes this case.
   \subsubsection*{Case \rm(ii) $|\la x \ra \cap K_0|=2$}
   Similarly with Case \rm(i), let $s\in \mathbb{F}_q^\times$ and $i$ be an integer such that $x^{g_0}=h(s)\phi^i.$
   Let $k=|\phi^i|$ and $q=q_1^k.$
   Thus,
   \begin{equation*}
   	(h(s)\phi^i)^k=h(ss^{\phi^i}s^{\phi^{2i}}\dots s^{\phi^{(k-1)i}})=h(\mathrm{N}_{\mathbb{F}_q/\mathbb{F}_{q_1}}(s))=h(-1),
   \end{equation*}
   and hence $\mathrm{N_{\mathbb{F}_q/\mathbb{F}_{q_1}}(s)}=-1.$
   On the other hand,
   $-1=(-1)^k=\mathrm{N}_{\mathbb{F}_q/\mathbb{F}_{q_1}}(-1),$ since $k$ is odd.
   Thus, there exists $\mu\in \mathbb{F}_q^\times$ such that $-s=\mu^{\phi^{-i}}\mu^{-1}$ by Hilbert 90.
   Therefore, 
   \begin{equation*}
   	x^{g_0h(\mu^{-1})}=h(\mu)h(s)\phi^ih(\mu^{-1})=h(\mu s\mu^{-\phi^{-i}})\phi^i=h(-1)\phi^i\leqslant (Q{:}H(q_0)){:}\la \phi \ra=N_G(E_0).
   \end{equation*}
   Therefore, this case is completed by taking $g=g_0h(\mu^{-1}).$
   
   \subsubsection*{Case \rm(iii) $|\la x \ra \cap K_0|>2$} 
   Let $g'\in G$ such that $x^{g_0g'}\in G_\omega,$ ($g'$ exists since $K$ is a fixer of $G$).
   Since $(G_0)_\omega$ contains a cyclic Hall subgroup of order $\frac{q_0-1}{2}$ by Lemma~\ref{Hall}, it follows that there exists a prime $p_1$ dividing $|\la x \ra \cap K_0|$ such that the Sylow $p_1$-subgroup of $H(q_0)$ is a Sylow subgroup of $(G_0)_\omega\cong {}^2G_2(q_0).$
   Let $D$ be the unique cyclic subgroup of order $p_1$ of $\la x \ra \cap K_0.$
   Then there exists $g_1\in (G_0)_\omega$ such that $D^{g_0g'g_1}\leqslant H(q_0),$ and hence 
   \begin{equation*}
   		x^{g_0g'g_1}\leqslant C_{G_\omega}(D^{g_0g'g_1})\leqslant H(q_0){:}\la \phi \ra, \, \text{and}\, 
   		(D^{g_0})^{g'g_1}=D^{g_0}\leqslant H(q_0)
   \end{equation*}
   by noting that $D^{g_0}\leqslant H(q_0).$
   Thus, $$g'g_1\in N_G(D^{g_0})=(H{:} \la \eta \ra){:}\la \phi \ra,$$
   where $$\eta=x_{a+b}(1)x_{-(a+b)}(1)x_{a+b}(1)x_{3a+b}(1)x_{-(3a+b)}(1)x_{3a+b}(1)$$ is the involution $n$ defined at \cite[Page 5]{matrixgen} (we use $\eta$ instead of $n$ here to avoid the abuse of notation).
   Note that there exists $g_2\in \{1, \eta\}$ such that $g'g_1g_2\in H{:}\la \phi \ra\leqslant N_G(Q),$ also note that $\eta$ normalizes $H(q_0)$ and commutes with $\phi.$
   Therefore, $$x^{g_0g'g_1g_2}\leqslant H(q_0){:}\la \phi \ra \leqslant (Q{:}H(q_0)){:}\la \phi \ra$$ with $g_0g'g_1g_2\in N_G(Q).$
   This case is completed by taking $g$ to be $g_0g'g_1g_2.$
   
   This completes the proof.      
	\end{proof}

\begin{lem}\label{lem:q27}
	If $q=27$ and $G_\omega$ is the normalizer of a cyclic Sylow $7$-subgroup of order $7,$ then the normalizer of a Sylow $2$-subgroup is indeed a fixer.
\end{lem}
\begin{proof}
	If $G=G_0,$ then this statement holds by Lemma~\ref{f:subfi}.
	Thus, $$G_0<G=\Aut(G_0)={}^2G_2(27){:}\la \phi \ra.$$
	Note that $C_{G_0}(\phi)\cong \PGammaL_2(8)$ contains the normalizer of a Sylow $2$-subgroup $A$ of $G_0,$ it follows that $N_G(A)=N_{G_0}(A)\times \la \phi \ra.$ Let $K=N_G(A).$ 
	
   We next prove that $K$ is indeed a fixer of $G.$
   To see this, let $k$ be an arbitrary element in $K,$ we will prove that $k$ is conjugate to an element in $G_\omega$ in $G.$
   Since $G_\omega$ is the normalizer of a Sylow $7$-subgroup, it suffice to show that $k$ normalizes a cyclic subgroup of order $7.$ 
   If $k\in K_0,$ then $k$ is conjugate to an element in $G_\omega$ in $G$ by Lemma~\ref{lem:reds} \rm(ii), we thus assume $k\notin K_0.$
   Thus, $kK_0$ is an element of order $3$ in $K/K_0,$ and the order of $k$ is one of $3,6$ and $21.$  
   We discuss these three cases one by one.
   
   Assume first the order of $k$ is 21. Then $k$ centralizes $\la k^3 \ra,$ and $|k^3|=7,$ thus $k$ is conjugate to an element in $G_\omega.$
   
   Assume next the order of $k$ is $3.$ Then there exists $g\in K$ such that $k^g\in \la t' \ra \times \la \phi \ra,$ where $t'$ is an element of order $3$ in $K_0.$ 
   Thus, $k^g$ normalizes a cyclic subgroup of order $7$ in $K_0,$ and hence $k^g$ is conjugate to some element in $G_\omega.$
   
   Assume finally the order of $k$ is $6.$
   Then there exists $g\in G$ such that $(k^2)^g$ is contained in $\la t^2 \ra \times \la \phi \ra,$ where $t$ is an element of order $6$ in $K.$
   Note that there exists $g_1\in G_G(\phi)\cong {}^2G_2(3)\cong \PGammaL_2(8)$ such that $(t^2)^{g_1}\in \la X_S(0,1,0) \ra$ by Lemma~\ref{centuni} and \ref{uni}. 
   Therefore, 
   \begin{equation*}
   	(k^2)^{gg_1}\in \la t^2 \ra^{g_1}\times \la \phi \ra=\la X_S(0,1,0) \ra \times \la \phi \ra.
   \end{equation*}
   Thus, $(k^2)^{gg_1}=X_S(0,a,0)\phi^i,$ such that $a \in \{\pm 1\}$ and $i\in \{1,2\}.$
   Since $$\mathrm{Tr}_{\mathbb{F}_{27}/\mathbb{F}_3}(a)=aa^{\phi}a^{\phi^2}=3a=0,$$
   it follows that $a=b-b^{\phi^{-i}}$ for some $b\in \mathbb{F}_{27}$ by additive form of Hilbert 90, and hence
   \begin{equation*}
   	\begin{aligned}
   	(k^2)^{gg_1X_S(0,b,0)}&=(X_S(0,a,0)\phi^i)^{X_S(0,b,0)}\\
   	&=X_S(0,-b,0)X_S(0,a,0)X_S(0,b^{\phi^{-i}},0)\phi^i=\phi^i. 
   \end{aligned}
   \end{equation*} 
  Thus, $(k^3)^{gg_1X_S(0,b,0)}$ is an involution in $C_{G_0}(\phi^i)\cong \PGammaL_2(8),$ and hence $(k^3)^{gg_1X_S(0,b,0)}$ normalizes a cyclic subgroup of order $7$ and $(k^2)^{gg_1X_S(0,b,0)}$ centralizes this cyclic subgroup of order $7$. 
  Thus, $k$ normalizes a cyclic subgroup of order $7,$ and hence $k$ is conjugate to an element in $G_\omega.$
  
  This completes the proof.
 
	\end{proof}

\subsection*{Proof of Theorem~\ref{thm:main}:}
Theorem~\ref{thm:main} is a combination of Lemma~\ref{lem:reds}, \ref{lem:subfield1}, \ref{lem:subfield2} and Lemma~\ref{lem:q27}. \qed

\end{document}